\documentclass{imamat}

\usepackage{graphicx,psfrag,epsf}
\usepackage{epstopdf}
\usepackage{subfigure}
\usepackage{algorithm}
\usepackage{algorithmicx}
\usepackage{algpseudocode}
\usepackage{multirow}
\usepackage{color}
\usepackage{geometry}

\usepackage{array,booktabs,tabularx}
\usepackage[latin1]{inputenc}
\usepackage{tabularx}

\newcommand{\imsize}{0.5}

\begin{document}

\title{Damped second order flow applied to image denoising}

\shorttitle{A damped second order flow} 
\shortauthorlist{G. Baravdish et al.} 

\author{
\name{G. Baravdish and O. Svensson}
\address{Department of Science and Technology, Link\"{o}ping University, 58183 Link\"{o}ping, Sweden}
\name{M. Gulliksson}
\address{School of Science and Technology, \"{O}rebro University, 70182 \"{O}rebro, Sweden}
\and
\name{Y. Zhang$^*$}
\address{Shenzhen MSU-BIT University, 518172 Shenzhen, China, and School of Mathematics and Statistics, Beijing Institude of Technology, 100081 Beijing, China\email{$^*$Corresponding author: ye.zhang@oru.se}}}

\maketitle

\begin{abstract}
{In this paper, we introduce a new image denoising model: the damped flow (DF), which is a second order nonlinear evolution equation associated with a class of energy functionals of an image. The existence, uniqueness, and regularization property of DF are proven. For the numerical implementation, based on the St\"{o}rmer-Verlet method, a discrete damped flow, SV-DDF, is developed. The convergence of SV-DDF is studied as well. Several numerical experiments, as well as a comparison with other methods, are provided to demonstrate the efficiency of SV-DDF.}
{nonlinear flow; image denoising; $p$-parabolic; $p$-Laplace; inverse problems; regularization; damped Hamiltonian system; symplectic method; St\"ormer-Verlet.}
\\
2000 Math Subject Classification: 35A01, 35A02, 65P10, 65M12, 65M32
\end{abstract}

\section{Introduction.}
\label{Introduction}

Digital images play a significant role in many fields in science, industry, and daily life, such as computer tomography, magnetic resonance imaging, geographical information systems, astronomy, satellite television, etc. Data sets collected by image sensors are always contaminated by noise. Instrument precision, the absence of some acquisition channels, and interfering natural phenomena can all degrade the data information. Moreover, noise can be introduced by transmission errors, compression and artificial editing. Therefore, it is necessary to apply a denoising technique on the original noisy image before it is analyzed.

Over the last few decades, scientists have developed numerous techniques to achieve adaptive imaging denoising, such as wavelets (\cite{Donoho-1995}), stochastic approaches (\cite{Preusser-2008}), and formulations based on partial differential equations (PDEs) (\cite{Alvarez1993,Weickert-1998}). We refer to (\cite{Gonzalez2007,Scherzer2009}) for a review on various denoising methods.

An essential challenge for imaging denoising is to remove noise as much as possible without eliminating the most representative characteristics of the image, such as edges, corners and other sharp structures. Traditional denoising methods are given some information about the noise, but the problem of blind image denoising involves computing the denoised image from the noisy one without any knowledge of the noise. The energy functional approach has in recent years been very successful in blind image denoising, most often taking the form
\begin{equation}
\mathcal{E} (u) = \frac{1}{2} \int_{\Omega} (u-u_0)^2 dx + \alpha \int_{\Omega} \Phi(|\nabla u|) dx,
\label{energyEu}
\end{equation}
where $u_0(x)$ is the observed (noisy) image, and $\Omega\subset \mathbf{R}^N$ ($N=2,3$) is a bounded domain with almost everywhere $C^{2}$ smooth boundary $\partial\Omega$. The first term in (\ref{energyEu}) is a fidelity term, the second term is a regularization term, and $\alpha>0$ is the regularization parameter. The regularization term $\Phi(|\nabla u|)$ is usually assumed to be strictly convex. A well-studied case of $\mathcal{E}(u)$ is when the regularization term is the $p$-Dirichlet energy, i.e.,
$\Phi(|\nabla u|)= \frac{1}{p} |\nabla u|^p, p\geq1$. The cases $p=1$ and $p=2$ correspond to the Total Variation (TV) principle (\cite{ROF}) and the first-order Tikhonov's regularization, respectively. The general cases for $p>1$ and $p=p(x)$ have been studied in, e.g., \cite{BOA2015,Kuijper:2009}.
Nowadays, there are many relevant extensions of the TV model, e.g.,
\cite{bollt2009graduated,bredies2010total,chen2006variable}.
The extension of TV to variational tensor-based formulations was investigated in
\cite{grasmair2010anisotropic}.
Other relevant extensions of the energy functional can be found in
\cite{aastrom2017}.

The Euler-Lagrange equation $\partial \mathcal{E} / \partial u =0$ associated with the functional $\mathcal{E}(u)$ is given by
\begin{equation}\label{ELeq}
u-u_0 - \alpha \cdot \textmd{div} \left( \frac{\Phi'(|\nabla u|)}{|\nabla u|} \nabla u \right) =0  \textmd{~in~} \  \Omega, \quad \partial_n u=0 \textmd{~on~} \ \partial\Omega,
\end{equation}
where $n$ is the outward unit normal to the boundary $\partial\Omega$. For the $p$-Dirichlet energy, let $\alpha\to0$, and we obtain the first order flow
\begin{equation}\label{PDEeq2}
u_t - \Delta_{p} u =  0  \textmd{~in~} \  (0,T)\times\Omega, \quad u(x,0) = u_0(x)  \textmd{~in~} \  \Omega,
\end{equation}
where the $p$-Laplace operator is defined by $\Delta_{p} u = \textmd{div}\left( |\nabla u|^{p-2} \nabla u \right)$. The $p$-parabolic equation in (\ref{PDEeq2}) has been studied intensively in
\cite{DiBenedettoBook,Ladyzhenskaja1988,lieberman1996,roubivcek2013,wu2001}, and references therein. The edge detection property of (\ref{PDEeq2}) has been analyzed in \cite{PeronaMalik1990}, where instead of the $p$-Laplace operator, they studied a more general diffusion term $\textmd{div} \left( g(|\nabla u|)\nabla u \right)$. The main advantage of (\ref{PDEeq2}) is the so-called conditional smoothing capability: For large $\nabla u$, the diffusion will be low, and therefore the exact localization of the edges will be kept. While $\nabla u(x)$ is small, the diffusion will tend to smooth around $x$. However, in the case of small noise of $u$ with large oscillations of the gradient $\nabla u$, the conditional smoothing introduced by (\ref{PDEeq2}) will keep all noise edges. To avoid the above mentioned difficulty, the authors in \cite{Alvarez1992} proposed a selective smoothing model
\begin{equation}\label{selective}
u_t - \textmd{div} \left( g(|\nabla G_{\sigma} \star u|)\nabla u \right) = 0,
\end{equation}
where $g$ is a smooth nonnegtive nonincreasing function with $g(0)=1$ and $\lim_{r\to +\infty} g(r)=0$. In (\ref{selective}), $G_{\sigma}(x)$, $\sigma>0$, is the Gaussian kernel $G_{\sigma}(x) = \frac{1}{(2\pi\sigma)^{N/2}}e^{-\frac{|x|^{2}}{2\sigma}}$, $|\nabla G_{\sigma}\star u|^2 = \sum_{j=1}^{N}\left(\frac{\partial G_\sigma}{\partial x_{j}}\star u\right)^{2}$, and $\star$ denote the cross-correlation, namely $(f\star g)(x) = \int_{\mathbf{R}^N} \bar{f}(y) g(x+y) dy$,  $\bar{f}$ is the complex conjugate of $f$.

On the other hand, for a better edge preservation, the authors in \cite{Ratner2011,Ratner2013} introduced a Telegraph-Diffusion (TeD) model, which is described by a second order (in time) hyperbolic equation
\begin{equation}\label{TeD}
u_{tt} + \eta u_t  - \textmd{div} \left( k \nabla u \right) = 0, \quad u(x,0) = u_0(x), u_t(x,0) = 0  \textmd{~in~} \  \Omega,
\end{equation}
where $\eta$ and $k$ denote the damping and elasticity coefficients, respectively. It has been shown that the model (\ref{TeD}) enables better preservation of edges in image denoising by offering an adaptive lowpass filter, and offering slower error propagation across edges.

Inspired by imaging denosing models (\ref{selective}) and (\ref{TeD}), in this paper, we will study the following second order flow
\begin{equation}\label{SecondOrderFlow}
\left\{\begin{array}{rl}
u_{tt}+\eta u_t - \textmd{div}\left( (\varepsilon + |\nabla G_{\sigma}                                                                                                                                                                                                                                                                                                                                                                                                                                                                                                                                                                                                                                                                                                                                                                                                                                                                                                                                                                                                                                                                                                                                                                                                                                                                                                                                                                                                                                                                                                                                                                                                                                                                                                                              \star u|^{2})^{\frac{p-2}{2}} \nabla u \right) =0 & \textmd{in} \  (0,T)\times\Omega,\\
u(x,0) = u_0(x) & \textmd{in} \  \Omega,\\
u_t(x,0) = 0 & \textmd{in} \  \Omega,\\
\partial_n u=0& \textmd{on} \ \partial\Omega.
\end{array}\right.
\end{equation}
where the damping parameter $\eta >0$ is a given model parameter, and $\sigma, \varepsilon>0$ are two anther given small numbers, which are used to avoid the singularity of model  (\ref{SecondOrderFlow}).

The model (\ref{SecondOrderFlow}) can be viewed as a regularized version of the telegraphers' equation with the $p$-Laplace operator, i.e. $u_{tt}+\eta u_t -\Delta_{p} u = 0$. Denote by $V(u)=\int_{\Omega} \frac{1}{p} |\nabla u|^p dx$ the $p$-Dirichlet integral. Then, the first order flow in (\ref{PDEeq2}), i.e. $u_t+\partial_u V(u)=0$, can be considered a classical steepest descent flow for solving the optimization problem $\min_u V(u)$. In the last two decades, there has been increasing evidence found showing that second order flows also enjoy remarkable optimization properties. Among these, a particularly important dynamical system~-- $u_{tt} + \eta u_t +\partial_u V(u) =0$ -- is called the Heavy Ball with Friction system (HBF) (\cite{Attouch-2000}) because of its mechanical interpretation. This system is an asymptotic approximation of the equation describing the motion of a material point with positive mass, subjected to stay on the graph of $V(u)$, and which moves under the action of the gravity force, the reaction force, and the friction force ($\eta>0$ is the friction parameter). The introduction of the term $\ddot{u}(t)$ in the dynamical system permits it to overcome some of the drawbacks of the steepest descent method. By contrast with steepest descent methods, the HBF system is not a descent method. It is the global energy (kinetic plus potential) which decreases. The optimization properties for the HBF system have been studied in detail in~\cite{Alvarez-2000,Alvarez-2002,Attouch-2000}, and references therein. Numerical algorithms based on the HBF system of solving some special problems, e.g. large systems of linear equations, eigenvalue problems, nonlinear Schr\"{o}dinger problems, inverse source problems, general ill posed problems, etc., can be found in~\cite{Edvardsson-2012,Edvardsson-2015,Sandin-2016,ZhangYe2018,ZhangHof2018,GongHofmannZhang2019}, where we can see that a second order damped system solved by a symplectic solver is far more efficient than numerically solving a first order system. In this study, we focus on the regularity of the specific system (\ref{SecondOrderFlow}) and its denoising capability.

The remainder of the paper is structured as follows. In Section 2, we study the existence and uniqueness of PDE (\ref{SecondOrderFlow}). Section 3 briefly discusses the regularization property of the dynamical solution with (\ref{SecondOrderFlow}). Based on the St\"{o}rmer-Verlet method, a discrete damped flow, termed by SV-DDF, is proposed in Section 4, where the convergence property of SV-DDF is studied. Section 5 presents an algorithm for image denoising. Several numerical examples are presented in Section 6 to demonstrate the feasibility and efficiency of the proposed method. A comparison with other methods is provided as well. Finally, concluding remarks are given in Section 7.

\section{Well-posedness of the model: Existence and uniquness results.}

We start with a brief description of the mathematical principles and some of the definitions used in this work. We denote by $H^{k}(\Omega)$, where $k$ is a positive integer, and the set of all functions $u$ defined in $\Omega$ is such that its distributional derivatives $D^{s}=\partial^{s}u/\partial x^{s}$ of order $|s|=\sum_{i=1}^{k}s_{i}\leq k$ all belong to
$L^{2}(\Omega)$. Furthermore, $H^{k}(\Omega)$ is a Hilbert space with the norm
$
\|u\|_{H^{k}(\Omega)}=\left(\sum_{s\leq k}\int_{\Omega}|D^{s}u|^{2}\,dx\right)^{1/2}.
$
The space $L^{p}(0,T;H^{k}(\Omega))$ consists of all functions $u$ such that for almost every $t\in(0,T)$, the element
$u(t)$ belongs to $H^{k}(\Omega)$. Hence, $L^{p}(0,T;H^{k}(\Omega))$ is a normed space with the norm
$
\|u\|_{L^{p}(0,T;H^{k}(\Omega))} = \left(\int_{0}^{T} \|u\|_{H^{k}(\Omega)}^{p}\,dt\right)^{1/p},
$
where $p>1$.
We also denote by $L^{\infty}(0,T;X)$ the set of all functions $u$ such that for almost every $t\in(0,T)$ the element
$u(t)$ belongs to $X$. $L^{\infty}(0,T;X)$ is a normed space with the norm
$
\|u\|_{L^{\infty}(0,T;X)} = \textmd{inf}\{C;\ \|u(t)\|_{X}\leq C,\mbox{ a.e. on } (0,T)\}.
$
We denote by $H^{1}(\Omega)^{*}$ the dual space of $H^{1}(\Omega)$. In the following, let $C_i$ denote a constant with a different value at a different place. It does not depend on the estimated quality. Moreover, to simplify the notation, we put
\begin{equation}\label{a}
a^{\varepsilon}(\nabla G_{\sigma}\star u) = (\varepsilon + |\nabla G_{\sigma}\star u|^{2})^{\frac{p-2}{2}} , \quad  p\in[1,2],
\end{equation}
and sometimes let $u_{t}(t)=u'(t)$.

Next, we introduce the solution space $\mathcal{U}$ for the problem (\ref{SecondOrderFlow}).
\begin{definition}
We say that an element $u$ belongs to the solution space $\mathcal{U}$ for the problem (\ref{SecondOrderFlow}) if $u\in L^{\infty}(0,T;H^{1}(\Omega))$, and its derivatives $u'$ and $u''$ with respect to $t$ in the sense of distributions to the spaces $L^{\infty}(0,T;L^{2}(\Omega))$ and $L^{\infty}(0,T;H^{1}(\Omega)^{*})$ respectively.
\end{definition}
It is easily seen that $\mathcal{U}$ is a Banach space equipped with the norm
$$
\| u \|_{\mathcal{U}} = \| u \|_{L^{\infty}(0,T;H^{1}(\Omega))} + \| u' \|_{L^{\infty}(0,T;L^{2}(\Omega))}
+ \| u'' \|_{L^{\infty}(0,T;H^{1}(\Omega)^{*})}.
$$
The solutions for the problem (\ref{SecondOrderFlow}) are considered in the weak sense, as follows.
\begin{definition}
A function $u$ is called a weak solution of the problem (\ref{SecondOrderFlow}) if $u\in \mathcal{U}$ and satisfies (\ref{SecondOrderFlow}) for almost every $t\in(0,+\infty)$ with derivatives of $u$ in the sense of distributions.
\end{definition}
We will show the existence of weak solutions for problem (\ref{SecondOrderFlow}) by using the Schauder fixed point theorem; see \cite{cao,catte1992image}.
In the sequel, we need the following results for the corresponding linear problem, \cite{Evans2010Book}, namely, 
\begin{equation}\label{LinPDE}
\left\{\begin{array}{rl}
u_{tt} + \eta u_t - \textmd{div}(b(x,t)\nabla u) =0 & \textmd{in} \  (0,T)\times\Omega,\\
u(x,0) = u_0(x) & \textmd{in} \  \Omega,\\
u_t(x,0) = 0 & \textmd{in} \  \Omega\\
\partial_n u=0& \textmd{on} \ \partial\Omega.
\end{array}\right.,
\end{equation}
where $b(x,t)$ is a given function such that $b(x,t)\geq c>0$ and $c$ is a constant.

\begin{theorem}\label{LemLinPDE}
Suppose that $b(x,t)$ is bounded, and let $u_{0}\in H^{1}(\Omega)$. Then the problem in (\ref{LinPDE}) has a unique solution $u\in{\mathcal U}$.
Moreover, if $u_{0}\in H^{2}(\Omega)$, then it follows that $u_{t}\in L^{\infty}(0,T;H^{1}(\Omega))$.
\end{theorem}
The linear problem (\ref{LinPDE}) is by now well-studied and Theorem \ref{LemLinPDE} can be proven by the Galerkin method, see \cite{Evans2010Book}. Before providing the main result, let us consider the following lemma.


\begin{lemma}\label{inequalitya}
Assume that for all $x\in\Omega$ and a.e. $t \in (0,T)$ there exists a positive constant $C$ depending on $G_{\sigma}$ and $\Omega$ such that
\begin{eqnarray}
||\nabla G_{\sigma}\star v||_{L^{\infty}(0,T;L^{\infty}(\Omega))}
          \leq C ||u_{0}||_{H^{1}(\Omega)}, \quad
||\nabla G_{\sigma}\star v_{t}||_{L^{\infty}(0,T;L^{\infty}(\Omega))}
          \leq C ||u_{0}||_{H^{1}(\Omega)}, \quad \forall v \in \mathcal{U}.
\end{eqnarray}

Then, the following inequalities hold for $p\in [1,2]$
\begin{eqnarray}
&& (\varepsilon + C^2 ||u_{0}||^2_{H^{1}(\Omega)})^{\frac{p-2}{2}} \leq a^{\varepsilon}(\nabla G_{\sigma}\star v) \leq \varepsilon^{\frac{p-2}{2}}, \quad \forall v \in \mathcal{U} \label{Ineqa} \\
&& |a_{t}^{\varepsilon}(\nabla G_{\sigma}\star v)| \leq (2-p)N C^2 \varepsilon^{\frac{p-4}{2}} ||u_{0}||^2_{H^{1}(\Omega)}, \quad \forall v \in \mathcal{U}. \label{Ineqat}
\end{eqnarray}
\end{lemma}

\begin{proof}
Inequalities (\ref{Ineqa}) hold obviously by noting the definition of $a^{\varepsilon}$ in (\ref{a}) and the following inequalities
\begin{eqnarray*}
\varepsilon^{\frac{2-p}{2}} \leq (\varepsilon + |\nabla G_{\sigma}\star v|^{2})^{\frac{2-p}{2}} \leq (\varepsilon + C^2 ||u_{0}||^2_{H^{1}(\Omega)})^{\frac{2-p}{2}}
\end{eqnarray*}
for $p\in[1,2]$, and for all $x\in\Omega$ and a.e. $t \in (0,T)$.

Now, consider the inequality (\ref{Ineqat}). Since
\begin{eqnarray*}
a_{t}^{\varepsilon}(\nabla G_{\sigma}\star v) = \frac{p-2}{2} \left( \varepsilon + |\nabla G_{\sigma}\star v|^{2} \right)^{\frac{p-4}{2}} \sum_{j=1}^{N} 2 \left(\frac{\partial G_\sigma}{\partial x_{j}}\star v \right) \left(\frac{\partial G_\sigma}{\partial x_{j}}\star v_t \right),
\end{eqnarray*}
we can deduce that
\begin{eqnarray*}
| a_{t}^{\varepsilon}(\nabla G_{\sigma}\star v)| &=& (2-p) \left( \varepsilon + |\nabla G_{\sigma}\star v|^{2} \right)^{\frac{p-4}{2}} \cdot \left|\sum_{j=1}^{N}  \left(\frac{\partial G_\sigma}{\partial x_{j}}\star v \right) \left( \frac{\partial G_\sigma}{\partial x_{j}}\star v_t \right) \right| \\
&\leq& (2-p) \varepsilon^{\frac{p-4}{2}} \cdot N \cdot C^2 ||u_{0}||^2_{H^{1}(\Omega)},
\end{eqnarray*}
by noting the conditions of the lemma.
\end{proof}

Now, we are in a position to show our main result.

\begin{theorem}\label{MainThm}
Assume that $p\in[1,2]$ and $u_{0}\in H^{2}(\Omega)$. Then a unique weak solution to problem (\ref{SecondOrderFlow}) exists if $T$ is sufficiently small with an upper bound depending on $||u_{0}||_{H^{1}(\Omega)}$, $G_{\sigma}$ and $\Omega$.
\end{theorem}

\begin{proof} {\bf Existence.} We use the Schauder fixed point theory, \cite{cao,catte1992image}, to prove the existence.
Let $v\in \mathcal{U}$ be such that
\begin{equation}\label{eq100}
||v||_{L^{\infty}(0,T;L^{2}(\Omega))} + ||v_{t}||_{L^{\infty}(0,T;L^{2}(\Omega))} \leq C_{1} ||u_{0}||_{H^{1}(\Omega)}
\end{equation}
where the positive constant $C_{1}$ will be determined later. Then the elements $\nabla G_{\sigma}\star v$ and $\nabla G_{\sigma} \star v_{t}$ belong to $L^{\infty}(0,T;C^{\infty}(\Omega))$, and
for all $x\in\Omega$ and a.e. $t \in (0,T)$
a positive constant $C_{2}$ depending on $G_{\sigma}$ and $\Omega$ exists such that
\begin{eqnarray*}
||\nabla G_{\sigma}\star v||_{L^{\infty}(0,T;L^{\infty}(\Omega))}
          \leq C_{1}C_{2} ||u_{0}||_{H^{1}(\Omega)}, \quad
||\nabla G_{\sigma}\star v_{t}||_{L^{\infty}(0,T;L^{\infty}(\Omega))}
          \leq C_{1}C_{2} ||u_{0}||_{H^{1}(\Omega)}.
\end{eqnarray*}
By Lemma \ref{inequalitya}, for all $x\in\Omega$ and a.e. $t$ in $(0,T)$, it follows that
\begin{eqnarray}
&& (\varepsilon +  C^2_{1}C^2_{2} ||u_{0}||^2_{H^{1}(\Omega)})^{\frac{p-2}{2}} \leq a^{\varepsilon}(\nabla G_{\sigma}\star v) \leq \varepsilon^{\frac{p-2}{2}}, \label{Ineqa2} \\
&& |a_{t}^{\varepsilon}(\nabla G_{\sigma}\star v)| \leq (2-p)N ( C_{1}C_{2})^2 \varepsilon^{\frac{p-4}{2}} ||u_{0}||^2_{H^{1}(\Omega)} . \label{Ineqat2}
\end{eqnarray}

Let $v\in \mathcal{U}$ satisfy (\ref{eq100}) and consider the problem $P_{v}$:
\begin{equation}\label{eq105}
\langle u_{tt},\varphi \rangle_{H^{1}(\Omega)^{*} \times H^{1}(\Omega)} +
\int_{\Omega} \left(\eta u_{t}\varphi+  a^{\varepsilon}(\nabla G_{\sigma}\star v)\nabla u \cdot \nabla\varphi\right)\,dx=0
\end{equation}
for every element $\varphi\in H^{1}(\Omega)$, a.e. $t$ in $(0,T)$.
The linear problem $P_{v}$ in (\ref{eq105}) is well-posed, \cite{Evans2010Book}, and has a solution $u_{v}$ which satisfies
\begin{equation}\label{EstimateUv}
||u_{v}||_{H^{1}(\Omega)} \leq C_{3} ||u_{0}||_{H^{1}(\Omega)},
\end{equation}
where $C_{3}$ is a positive constant, depending only on the domain $\Omega$.

Now, let us consider two cases when $\varphi=u_{v}'$ and $\phi = u_{\nu}$.

{\bf Case 1.} If $\varphi=u_{v}'$, we have
$$
\langle u_{v}'',u_{v}' \rangle_{H^{1}(\Omega)^{*} \times H^{1}(\Omega)} +
\int_{\Omega}  \eta (u_{v}')^{2}  \,dx+
\int_{\Omega} a^{\varepsilon}(\nabla G_{\sigma}\star v)\nabla u_{v} \cdot \nabla u_{v}'  \,dx =0,
$$
which gives
$$
 \frac{1}{2}  \frac{d}{dt}\langle u_{v}',u_{v}' \rangle_{H^{1}(\Omega) \times H^{1}(\Omega)}
 + \eta  \int_{\Omega} (u_{v}')^{2}  \,dx +
\int_{\Omega} a^{\varepsilon}(\nabla G_{\sigma}\star v)\nabla u_{v} \cdot \nabla u_{v}'\,dx =0,
$$
or equivalently
$$
\frac{d}{dt} ||u_{v}'||^{2}_{L^{2}(\Omega)} =
    - \frac{d}{dt}||\nabla u_{v}'||^{2}_{L^{2}(\Omega)}
-2\eta ||u_{v}'||^{2}_{L^{2}(\Omega)}
-2\int_{\Omega} a^{\varepsilon}(\nabla G_{\sigma}\star v)\nabla u_{v} \cdot \nabla u_{v}'\,dx
$$
Denote by $I_{1}=-2\int_{\Omega} a^{\varepsilon}(\nabla G_{\sigma}\star v)\nabla u_{v} \cdot \nabla u_{v}'\,dx$ and, integrating the above equation, we obtain
$$
||u_{v}'(t)||^{2}_{L^{2}(\Omega)} = -||\nabla u_{v}'(t)||^{2}_{L^{2}(\Omega)}
-2\eta \int_{0}^{t}||u_{v}'||^{2}_{L^{2}(\Omega)}\,d\tau
+ \int_{0}^{t}I_{1}\,d\tau,
$$
where we have used that $u_{vt}(x,0)=0$ and $\nabla u_{vt}(x,0)=0$. Hence,
\begin{equation}\label{eq110}
||u_{v}'(t)||^{2}_{L^{2}(\Omega)} \leq -2\eta \int_{0}^{t}||u_{v}'||^{2}_{L^{2}(\Omega)}\,d\tau
+ \int_{0}^{t}I_{1}\,d\tau
\end{equation}

On the other hand, by the inequalities (\ref{Ineqa2}), (\ref{Ineqat2}) and (\ref{EstimateUv}), we have
\begin{eqnarray}\label{Ineq1}
&&\int_{0}^{t}\int_{\Omega}a_{t}^{\varepsilon}(\nabla G_{\sigma}\star v)\nabla u_{v} \cdot \nabla u_{v}\,dx\,d\tau  \leq (2-p)N ( C_{1}C_{2})^2 \varepsilon^{\frac{p-4}{2}} ||u_{0}||^2_{H^{1}(\Omega)} T ||\nabla u_{v}||^{2}_{L^{2}(\Omega)} \nonumber \\ && \quad \leq (2-p)NT( C_{1}C_{2})^2 \varepsilon^{\frac{p-4}{2}} ||u_{0}||^2_{H^{1}(\Omega)} ||u_{v}||^{2}_{H^{1}(\Omega)}
\leq (2-p)NT( C_{1}C_{2}C_{3})^2 \varepsilon^{\frac{p-4}{2}} ||u_{0}||^4_{H^{1}(\Omega)}
\end{eqnarray}
and
\begin{eqnarray}\label{Ineq2}
&& -\int_{0}^{t}\int_{\Omega}\frac{d}{dt} [a^{\varepsilon}(\nabla G_{\sigma}\star v)\nabla u_{v} \cdot \nabla u_{v}]\,dx\,dt \nonumber \\ && \qquad =
[a^{\varepsilon}(\nabla G_{\sigma}\star v)\nabla u_{v} \cdot \nabla u_{v}]_{\tau=0}
-[a^{\varepsilon}(\nabla G_{\sigma}\star v)\nabla u_{v} \cdot \nabla u_{v}]_{\tau=t} \leq 2 \varepsilon^{\frac{p-2}{2}} C^2_{3} ||u_{0}||^{2}_{H^{1}(\Omega)}.
\end{eqnarray}

Since
\begin{eqnarray*}
I_{1} =  -2\int_{\Omega} a^{\varepsilon}(\nabla G_{\sigma}\star v)\nabla u_{v} \cdot \nabla u_{v}'\,dx = \int_{\Omega}a_{t}^{\varepsilon}(\nabla G_{\sigma}\star v)\nabla u_{v} \cdot \nabla u_{v}\,dx
-\int_{\Omega}\frac{d}{dt} [a^{\varepsilon}(\nabla G_{\sigma}\star v)\nabla u_{v} \cdot \nabla u_{v}]\,dx,
\end{eqnarray*}
by combining (\ref{Ineq1}) and (\ref{Ineq2}), one can deduce that
\begin{equation}\label{Ineq3}
\int_{0}^{t}I_{1}\,d\tau \leq (2-p)NT( C_{1}C_{2}C_{3})^2 \varepsilon^{\frac{p-4}{2}} ||u_{0}||^4_{H^{1}(\Omega)} + 2 \varepsilon^{\frac{p-2}{2}} C^2_{3} ||u_{0}||^{2}_{H^{1}(\Omega)}.
\end{equation}

Furthermore, using Gr\"onwall's lemma in (\ref{eq110}), we get
\begin{equation}\label{eq115}
||u_{v}'(t)||^{2}_{L^{2}(\Omega)} \leq \int_{0}^{t}I_{1}\,d\tau +
\int_{0}^{t}\left[\left( \int_{0}^{\tau}I_{1}\,ds\right)(-2\eta)e^{\int_{0}^{\tau}(-2\eta)\,ds}\right]\,d\tau .
\end{equation}
The second term in the right-hand side of the above inequality can be rewritten as
\begin{eqnarray*}
2\eta \int_{0}^{t}\left[ \int_{0}^{\tau} \left(
\int_{\Omega}\frac{d}{dt} [a^{\varepsilon}(\nabla G_{\sigma}\star v)\nabla u_{v} \cdot \nabla u_{v}]\,dx
-\int_{\Omega}a_{t}^{\varepsilon}(\nabla G_{\sigma}\star v)\nabla u_{v} \cdot \nabla u_{v}\,dx
\right)\,ds\right] \cdot e^{-2\eta\tau}\,d\tau = J_{1}+J_{2}.
\end{eqnarray*}

It follows that
\begin{eqnarray*}
|J_{1}| &=&2\eta \int_{0}^{t}\left[
\int_{\Omega} a^{\varepsilon}(\nabla G_{\sigma}\star v)\nabla u_{v} \cdot \nabla u_{v}\,dx\right]_{s=0}^{s=\tau}e^{-2\eta\tau}\,d\tau \\
&\leq&
4\eta \int_{0}^{t}\left[ \varepsilon^{\frac{p-2}{2}} C^{2}_{3} ||u_{0}||^{2}_{H^{1}(\Omega)} \right] e^{-2\eta\tau}\,d\tau \leq
2 \varepsilon^{\frac{p-2}{2}} C^{2}_{3} ||u_{0}||^{2}_{H^{1}(\Omega)}
\end{eqnarray*}
by noting inequality (\ref{Ineqa2}). Similarly, by (\ref{Ineqat2}) and (\ref{EstimateUv}), one can deduce that
\begin{eqnarray*}
|J_{2}| &=&  2\eta \int_{0}^{t}\left[ \int_{0}^{\tau} \left(\int_{\Omega}a_{t}^{\varepsilon}(\nabla G_{\sigma}\star v)\nabla u_{v} \cdot \nabla u_{v}\,dx
\right)\,ds\right]e^{-2\eta\tau}\,d\tau \\
&\leq&  2\eta \int_{0}^{t}\left[ \tau (2-p)N ( C_{1}C_{2})^2 \varepsilon^{\frac{p-4}{2}} ||u_{0}||^2_{H^{1}(\Omega)} ||\nabla u_{v}||^2_{L^{2}(\Omega)} \right]e^{-2\eta\tau}\,d\tau \\
&\leq&  2(2-p)N \eta ( C_{1}C_{2}C_{3} )^2 \varepsilon^{\frac{p-4}{2}} ||u_{0}||^4_{H^{1}(\Omega)} \int_{0}^{t}\tau e^{-2\eta\tau}\,d\tau \\
&\leq& 2(2-p)N \eta ( C_{1}C_{2}C_{3} )^2 \varepsilon^{\frac{p-4}{2}} ||u_{0}||^4_{H^{1}(\Omega)}.
\end{eqnarray*}

Putting the above two inequalities for $|J_i|$, $i=1,2$, and the inequality (\ref{Ineq3}) in the estimate  (\ref{eq115}), we derive
\begin{eqnarray*}
||u_{v}'(t)||^{2}_{L^{2}(\Omega)} &\leq&
(2-p)NT( C_{1}C_{2}C_{3} )^2 \varepsilon^{\frac{p-4}{2}} ||u_{0}||^4_{H^{1}(\Omega)} + 2 \varepsilon^{\frac{p-2}{2}} C^2_{3} ||u_{0}||^{2}_{H^{1}(\Omega)} \\
&& + 2 \varepsilon^{\frac{p-2}{2}} C^2_{3} ||u_{0}||^{2}_{H^{1}(\Omega)} + 2(2-p)N \eta ( C_{1}C_{2}C_{3} )^2 \varepsilon^{\frac{p-4}{2}} ||u_{0}||^4_{H^{1}(\Omega)}.
\end{eqnarray*}

Finally, denote by
\begin{eqnarray}\label{eq125}
&& C= 4 \varepsilon^{\frac{p-2}{2}} C^2_{3} + (2-p)N( C_{1}C_{2}C_{3})^2 \varepsilon^{\frac{p-4}{2}} ||u_{0}||^2_{H^{1}(\Omega)} (T + 2\eta),
\end{eqnarray}
and we obtain
\begin{equation}\label{eq120}
||u_{v}'||^{2}_{L^{2}(\Omega)} \leq C ||u_{0}||^{2}_{H^{1}(\Omega)}.
\end{equation}

{\bf Case 2.} If $\varphi=u_{v}$ in (\ref{eq105}), we have
$$
\langle u_{v}'',u_{v} \rangle_{H^{1}(\Omega)^{*}\times H^{1}(\Omega)}+
\int_{\Omega}  \eta u_{v}' u_{v}  \,dx+
\int_{\Omega} a^{\varepsilon}(\nabla G_{\sigma}\star v)\nabla u_{v} \cdot \nabla u_{v}  \,dx =0,
$$
which is equivalent to
$$
\frac{1}{2}\frac{d^{2}}{dt^{2}} ||u_{v}||^{2}_{H^{1}(\Omega)}
- ||u_{v}'||^{2}_{L^{2}(\Omega)}
+ \frac{1}{2}\eta \int_{\Omega} \frac{d}{dt}|u_{v}|^{2}\,dx
+ \int_{\Omega} a^{\varepsilon}(\nabla G_{\sigma}\star v)\nabla u_{v} \cdot \nabla u_{v}\,dx
=0
$$
by noting
$$
\langle u_{v}'',u_{v} \rangle_{H^{1}(\Omega)^{*} \times H^{1}(\Omega)}
=
\frac{1}{2}\frac{d^{2}}{dt^{2}} ||u_{v}||^{2}_{H^{1}(\Omega)}
- ||u_{v}'||^{2}_{L^{2}(\Omega)}.
$$
Denote $I_{2}=-2\int_{\Omega} a^{\varepsilon}(\nabla G_{\sigma}\star v)\nabla u_{v} \cdot \nabla u_{v}\,dx$, and we obtain
$$
\frac{d^{2}}{dt^{2}} ||u_{v}||^{2}_{H^{1}(\Omega)}
= 2||u_{v}'||^{2}_{L^{2}(\Omega)}
- \eta \frac{d}{dt}||u_{v}||^{2}_{L^{2}(\Omega)}
+ I_{2}
$$
We integrate the above identity twice and get
\begin{eqnarray}\label{Equv}
||u_{v}||^{2}_{H^{1}(\Omega)}
=
2\int_{0}^{t}\int_{0}^{\tau}||u_{v}'||^{2}_{L^{2}(\Omega)}\,ds\,d\tau
- \eta \int_{0}^{t} (||u_{v}(\tau)||^{2}_{L^{2}(\Omega)} - ||u_{v}(0)||^{2}_{L^{2}(\Omega)} ) \,d\tau + \int_{0}^{t}\int_{0}^{\tau} I_{2}\,ds\,d\tau
\end{eqnarray}
or equivalently
\begin{eqnarray}\label{Equv2}
&& ||u_{v}||^{2}_{L^{2}(\Omega)}
= - ||\nabla u_{v}||^{2}_{L^{2}(\Omega)} \nonumber \\
&& \quad
+ 2\int_{0}^{t}\int_{0}^{\tau}||u_{v}'||^{2}_{L^{2}(\Omega)}\,ds\,d\tau- \eta \int_{0}^{t} ||u_{v}(\tau)||^{2}_{L^{2}(\Omega)} \,d\tau
+ t\eta ||u_0||^{2}_{L^{2}(\Omega)}
+ \int_{0}^{t}\int_{0}^{\tau} I_{2}\,ds\,d\tau .
\end{eqnarray}

By inequalities (\ref{Ineqa2}) and (\ref{EstimateUv}), we have
$
|I_{2}|= 2\int_{\Omega} a^{\varepsilon}(\nabla G_{\sigma}\star v)\nabla u_{v} \cdot \nabla u_{v}\,dx
\leq
2 \varepsilon^{\frac{p-2}{2}} C^2_{3} ||u_{0}||^{2}_{H^{1}(\Omega)}.
$
Furthermore, it follows by (\ref{eq120}) that
\begin{equation}\label{EstimateUvprime}
\sup_{0\leq t\leq T}||u_{v}'(t)||^{2}_{L^{2}(\Omega)} \leq C ||u_{0}||^{2}_{H^{1}(\Omega)}.
\end{equation}
Hence, by ignoring the non-positive terms in the right-hand side of (\ref{Equv2}) we obtain
\begin{eqnarray}\label{Inequ}
||u_{v}||^{2}_{L^{2}(\Omega)} &\leq&
 2T^{2}C ||u_{0}||^{2}_{H^{1}(\Omega)} + \eta T ||u_{0}||^{2}_{H^{1}(\Omega)} +
2 \varepsilon^{\frac{p-2}{2}} C^2_{3} T^2 ||u_{0}||^{2}_{H^{1}(\Omega)} \nonumber \\
 &=&
 T \left( 2 TC  + \eta   +
2 \varepsilon^{\frac{p-2}{2}} C^2_{3} T \right) ||u_{0}||^{2}_{H^{1}(\Omega)}.
\end{eqnarray}

Inserting $u_{v}$ into (\ref{eq100}) to obtain
\begin{equation}\label{eq205}
||u_{v}||_{L^{\infty}(0,T;L^{2}(\Omega))} + ||u_{vt}||_{L^{\infty}(0,T;L^{2}(\Omega))} \leq
   C_{1} ||u_{0}||_{H^{1}(\Omega)}.
\end{equation}

By combining (\ref{eq120}), (\ref{Inequ}) and (\ref{eq205}), we obtain that
\begin{equation}\label{eq205}
||u_{v}||_{L^{\infty}(0,T;L^{2}(\Omega))} + ||u_{vt}||_{L^{\infty}(0,T;L^{2}(\Omega))} \leq
   C_{1} ||u_{0}||_{H^{1}(\Omega)}
\end{equation}
provided that
$$
\sqrt{2 T^{2} C  + \eta T  + 2 \varepsilon^{\frac{p-2}{2}} C^2_{3} T^{2} } + \sqrt{C} \leq C_{1}.
$$

Hence, it is sufficient to show that there exists $C_{1}>0$ such that
\begin{eqnarray*}\label{C1requirment}
2T^{2}C  + \eta T  + 2 \varepsilon^{\frac{p-2}{2}} C^2_{3} T^{2} + C \leq C^2_{1},
\end{eqnarray*}
which is equivalent to
\begin{eqnarray*}
\eta T + 2 (5 T^2 + 2) \varepsilon^{\frac{p-2}{2}} C^2_{3} \leq \left[ 1 - (2-p)N(C_{2}C_{3})^2 \varepsilon^{\frac{p-4}{2}} ||u_{0}||^2_{H^{1}(\Omega)} (2T^2 +1) (T + 2\eta ) \right] C^2_1
\end{eqnarray*}
by noting the definition of $C$ in (\ref{eq125}). Hence, we have to require that
\begin{eqnarray*}
(2-p)N(C_{2}C_{3})^2 \varepsilon^{\frac{p-4}{2}} ||u_{0}||^2_{H^{1}(\Omega)} (2T^2 +1) (T + 2\eta ) < 1.
\end{eqnarray*}
This is easily fulfilled if we choose
\begin{eqnarray*}
T < \frac{1}{\sqrt[3]{2(2-p)N(C_{2}C_{3})^2} \varepsilon^{\frac{p-4}{6}} ||u_{0}||^{2/3}_{H^{1}(\Omega)}}.
\end{eqnarray*}

On the other hand, by equation (\ref{Equv}) we obtain
\begin{eqnarray*}
||\nabla u_{v}||^{2}_{L^{2}(\Omega)}
&=&
- ||u_{v}||^{2}_{L^{2}(\Omega)}
+ 2\int_{0}^{t}\int_{0}^{\tau}||u_{v}'||^{2}_{L^{2}(\Omega)}\,ds\,d\tau \\
&& - \eta \int_{0}^{t} ||u_{v}(\tau)||^{2}_{L^{2}(\Omega)} \,d\tau
+ t\eta ||u_{v}(0)||^{2}_{L^{2}(\Omega)}
+ \int_{0}^{t}\int_{0}^{\tau} I_{2}\,ds\,d\tau.
\end{eqnarray*}
Since $ I_{2}=-2\int_{\Omega} a^{\varepsilon}(\nabla G_{\sigma}\star v)\nabla u_{v} \cdot \nabla u_{v}\,dx
\leq -2 \varepsilon ||\nabla u_{v}||^{2}_{L^{2}(\Omega)} \leq 0$, using inequalities (\ref{EstimateUvprime}) and (\ref{Inequ}), we can deduce that
\begin{eqnarray}\label{IneqduvL2}
 ||\nabla u_{v}||^{2}_{L^{2}(\Omega)} &\leq&
2\int_{0}^{t}\int_{0}^{\tau}||u_{v}'||^{2}_{L^{2}(\Omega)}\,ds\,d\tau
+ t\eta ||u_{v}(0)||^{2}_{L^{2}(\Omega)} \nonumber \\ &\leq& 2T^2 C ||u_{0}||^{2}_{H^{1}(\Omega)} + \eta T \left( 2T^{2}C  + \eta T  +
2 \varepsilon^{\frac{p-2}{2}} C^2_{3} T^{2} \right) ||u_{0}||^{2}_{H^{1}(\Omega)} \nonumber \\ &=& \left( 2(1+\eta T^2) T^{2}C  + \eta^2 T^2  + 2 \varepsilon^{\frac{p-2}{2}} \eta T^3 C^2_{3} \right) ||u_{0}||^{2}_{H^{1}(\Omega)}
\end{eqnarray}
a.e. $t \in (0,T)$. Define $C_{4}= \sqrt{2(1+\eta T^2) T^{2}C  + \eta^2 T^2  + 2 \varepsilon^{\frac{p-2}{2}} \eta T^3 C^2_{3}} \cdot ||u_{0}||_{H^{1}(\Omega)}$, and note that $C_{4}$ is independent on $t$, and we obtain
\begin{equation}\label{eq230a}
||\nabla u_{v}||_{L^{\infty}(0,T;L^{2}(\Omega))} \leq C_{4} .
\end{equation}

Now, let $\varphi$ in (\ref{eq105}) such that $||\varphi||_{H^{1}(\Omega)} = 1$, and we obtain
\begin{eqnarray*}
&&\langle u_{v}'', \varphi \rangle_{H^{1}(\Omega)^{*} \times H^{1}(\Omega)} =
- \eta \int_{\Omega}   u_{v}' \varphi \,dx -
\int_{\Omega} a^{\varepsilon}(\nabla G_{\sigma}\star v)\nabla u_{v} \cdot \nabla \varphi  \,dx \\
&&  \leq - \eta \int_{\Omega}   u_{v}' \varphi \,dx -
       (\varepsilon + C^2 ||u_{0}||^2_{H^{1}(\Omega)})^{\frac{p-2}{2}} \int_{\Omega} \nabla u_{v} \cdot \nabla \varphi  \,dx \\
&&  \leq \eta ||u_{v}'||_{L^{2}(\Omega)} ||\varphi||_{L^2(\Omega)}
 + (\varepsilon + C^2 ||u_{0}||^2_{H^{1}(\Omega)})^{\frac{p-2}{2}} ||\nabla u_{v}||_{L^{2}(\Omega)} ||\nabla \varphi||_{L^2(\Omega)} \\
&&  \leq ||u_{0}||^{2}_{H^{1}(\Omega)} \cdot \Big( \eta C  +  (\varepsilon + C^2 ||u_{0}||^2_{H^{1}(\Omega)})^{\frac{p-2}{2}} \cdot \left( 2(1+\eta T^2) T^{2}C  + \eta^2 T^2  + 2 \varepsilon^{\frac{p-2}{2}} \eta T^3 C^2_{3} \right) \Big) =: C_5.
\end{eqnarray*}
by noting inequalities (\ref{eq120}) and (\ref{IneqduvL2}). This implies that
$$
||u_{v}''||_{H^{1}(\Omega)^{*}} =
\sup_{||\varphi||_{H^{1}(\Omega)} =1} \langle u_{v}'', \varphi \rangle_{H^{1}(\Omega)^{*} \times H^{1}(\Omega)}  \leq C_{5}.
$$
Since constant $C_{5}$ is independent on $t$, we obtain
\begin{equation}\label{eq230b}
||u_{v}''||_{L^{\infty}(0,T;H^{1}(\Omega)^{*})} \leq  C_{5}.
\end{equation}

From (\ref{eq205}), (\ref{eq230a}) and (\ref{eq230b}), we introduce the subspace $\mathcal{U}_{0}$ of $\mathcal{U}$ defined by
\begin{eqnarray*}
&& \mathcal{U}_{0}=\left\{
v\in\mathcal{U}; v~\mbox{satisfies (\ref{SecondOrderFlow}) in the sense of distribution},\right.\\
&& ||v||_{L^{\infty}(0,T;L^{2}(\Omega))} + ||v'||_{L^{\infty}(0,T;L^{2}(\Omega))} \leq
   C_{1} ||u_{0}||_{H^{1}(\Omega)}, ~ ||\nabla u_{v}||_{L^{\infty}(0,T;L^{2}(\Omega))} \leq C_{4}, ~
||u_{v}''||_{L^{\infty}(0,T;H^{1}(\Omega)^{*})} \leq C_{5} \}
\end{eqnarray*}
It follows by construction that $P: v\rightarrow u_{v}$ is a mapping from $\mathcal{U}_{0}$ to $\mathcal{U}_{0}$.
Furthermore, it can be shown that $\mathcal{U}_{0}$ is a nonempty, convex and weakly compact subset of
$\mathcal{U}$. We want to use Schauder's fixed point theorem and need to prove that $P: v\rightarrow u_{v}$, with a weakly continuous mapping from $\mathcal{U}_{0}$ to $\mathcal{U}_{0}$.
Let $v_{j}$ be a sequence that converges weakly to some $v$ in $\mathcal{U}_{0}$ and let $u_{j}=u_{v_{j}}$.
We have to prove that $u_{j}=P(v_{j})$ converges weakly to $u_{v}=P(v)$.
From (\ref{eq230a}) and (\ref{eq230b}), classical results of compact inclusion in Sobolev spaces, \cite{adams2003sobolev},
we can select from $v_{j}$ and $u_{j}$, respectively, a subsequence such that for some $u$, we have
\begin{itemize}
\item $v_{j}\rightarrow v$ in $L^{2}(0,T;L^{2}(\Omega))$ and a.e. on $\Omega\times(0,T)$,
\item $\partial_{x_{k}}G_{\sigma} \star v_{j} \rightarrow \partial_{x_{k}}G_{\sigma} \star v$ in
$L^{2}(0,T;L^{2}(\Omega))$ and a.e. on $\Omega\times(0,T)$, $k=1,2,\ldots, N$,
\item $a^{\varepsilon}(|\partial_{x_{k}}G_{\sigma} \star v_{j}|) \rightarrow
           a^{\varepsilon}(|\partial_{x_{k}}G_{\sigma} \star v|)$
           in $L^{2}(0,T;L^{2}(\Omega)$ and a.e. on $\Omega\times(0,T)$,
\item $u_{j}\rightarrow u$ weakly $*$ in $L^{\infty}(0,T;H^{1}(\Omega))$,
\item $u_{j}'\rightarrow u'$ weakly $*$ in $L^{\infty}(0,T;L^{2}(\Omega))$,
\item $u_{j}''\rightarrow u''$ weakly $*$ in $L^{2}(0,T;H^{1}(\Omega)^*)$,
\item $u_{j}\rightarrow u$ in $L^{2}(0,T;L^{2}(\Omega))$ and a.e. on $\Omega\times(0,T)$,
\item $\frac{\partial u_{j}}{\partial x_{k}}\rightarrow \frac{\partial u}{\partial x_{k}}$ weakly
          $*$ in $L^{\infty}(0,T;L^{2}(\Omega))$, $k=1,2,\ldots, N$,
\item $u_{j}(0) \rightarrow u_{0}$ in $L^{2}(\Omega)$,
\item $u_{j}'(0)\rightarrow 0$ in $H^{1}(\Omega)^{*}$.
\end{itemize}
Hence, we can define $u=P(v)$ as the limit in the problem $P_{v_{j}}$.
Moreover, $P$ is weakly continuous, since the
sequence $u_{j}=P(v_{j})$ converges weakly in
$\mathcal{U}_{0}$ to a unique element $u=P(v)$.
By  the Schauder fixed point theorem, there exists $v\in\mathcal{U}_{0}$ such that $v=P(v)=u_{v}$ showing that the element $u_{v}$ solves the problem (\ref{SecondOrderFlow}).

\vspace{0.2mm}


{\bf Uniqueness.} We proceed as in \cite{Evans2010Book,cao}. Let $u_{1}$ and $u_{2}$ be two weak solutions of
(\ref{SecondOrderFlow}). Denote by $a_i = a^{\varepsilon}(\nabla G_{\sigma}\star u_{i})$, where $a^{\varepsilon}(\cdot)$ is defined in (\ref{a}). Then for a.e. $t\in(0,T)$ we obtain
\begin{equation}\label{eq300}
(u_{1}-u_{2})''+\eta (u_{1}-u_{2})' -
\textmd{div}(a_1 \nabla (u_{1}-u_{2})) =\textmd{div}((a_{1}-a_{2}) \nabla u_{2}),
\end{equation}
subject to the initial condition
\begin{equation}\label{eq305}
(u_{1}-u_{2})(x,0)=0,~ (u_{1}-u_{2})_{t}(x,0)=0, \quad x\in\Omega,
\end{equation}
and the boundary condition
\begin{equation}\label{eq310}
\partial_n (u_{1}-u_{2})=0, \quad x\in \partial\Omega, \quad t\in(0,T)
\end{equation}
in the distribution sense.

It suffices to show that $u_{1}-u_{2}=0$. Now, fix $s\in(0,t)$ and let
$$
v_{k}(t)=\left\{\begin{array}{ll}
\int_{t}^{s} u_{k}(\tau)\,d\tau,&0<t\leq s\\0,&s\leq t<T
\end{array}\right.
$$
for $k=1,2$. Then for every $t\in(0,T)$, we have that $v_{k}(t)\in H^{1}(\Omega)$ and $\partial_{n}v_{k}=0$ on $\partial\Omega$.
Multiplying (\ref{eq300}) by $v_{1}-v_{2}$ and integrating, we obtain
\begin{eqnarray*}
\int_{0}^{s}\int_{\Omega}(u_{1}-u_{2})''(v_{1}-v_{2})  &+& \eta (u_{1}-u_{2})'(v_{1}-v_{2}) -
\textmd{div}(a_{1} \nabla (u_{1}-u_{2}))(v_{1}-v_{2}) \,dx\,dt\\
&=& \int_{0}^{s}\int_{\Omega}\textmd{div}((a_{1}-a_{2}) \nabla u_{2})(v_{1}-v_{2})\,dx\,dt.
\end{eqnarray*}

Applying integration by parts with respect to the time variable, we obtain
\begin{eqnarray*}
&& -\int_{0}^{s}\int_{\Omega}(u_{1}-u_{2})' (v_{1}-v_{2})' \,dx\,dt
- \eta \int_{0}^{s}\int_{\Omega} (u_{1}-u_{2})(v_{1}-v_{2})' \,dx\,dt
\\
&& \qquad\qquad + \int_{0}^{s}\int_{\Omega}(a_{1} \nabla (u_{1}-u_{2})) \cdot \nabla(v_{1}-v_{2}) \,dx\,dt  = - \int_{0}^{s}\int_{\Omega} (a_{1}-a_{2}) \nabla u_{2}\cdot \nabla(v_{1}-v_{2})\,dx\,dt
\end{eqnarray*}
by noting the initial condition (\ref{eq305}) and the fact $v_{k}(s) \equiv 0$, $k=1, 2$.

If we set $v_{k}'=-u_{k}$ in the above equation, we obtain
\begin{eqnarray*}
&& \int_{0}^{s}\int_{\Omega}(u_{1}-u_{2})' (u_{1}-u_{2}) \,dx\,dt
+ \eta \int_{0}^{s}\int_{\Omega} |u_{1}-u_{2}|^{2} \,dx\,dt \\ && \qquad\qquad
+ \int_{0}^{s}\int_{\Omega}(a_{1} \nabla (v_{1}-v_{2})')\cdot \nabla (v_{1}-v_{2}) \,dx\,dt
= - \int_{0}^{s}\int_{\Omega} (a_{1}-a_{2}) \nabla u_{2}\cdot \nabla(v_{1}-v_{2})\,dx\,dt,
\end{eqnarray*}
or equivalently
\begin{eqnarray*}
&& \frac{1}{2}\int_{\Omega} |u_{1}-u_{2}|^{2} \,dx
+ \eta \int_{0}^{s}\int_{\Omega} |u_{1}-u_{2}|^{2} \,dx\,dt - \frac{1}{2}\int_{\Omega} a_{1} |\nabla (v_{1}-v_{2})|^{2}|_{t=s} \,dx + \frac{1}{2}\int_{\Omega} a_{1} |\nabla (v_{1}-v_{2})|^{2}|_{t=0} \,dx \\
&& \qquad + \frac{1}{2}\int_{0}^{s}\int_{\Omega} a_{1}' |\nabla (v_{1}-v_{2})|^{2} \,dx\,dt
= - \int_{0}^{s}\int_{\Omega} (a_{1}-a_{2}) \nabla u_{2}\cdot \nabla(v_{1}-v_{2})\,dx\,dt
\end{eqnarray*}
Since $\nabla v_{k}(s) \equiv0$, one can deduce that
\begin{eqnarray*}
&&\frac{1}{2}\int_{\Omega} |u_{1}-u_{2}|^{2} \,dx
+ \eta \int_{0}^{s}\int_{\Omega} |u_{1}-u_{2}|^{2} \,dx\,dt
 + \frac{1}{2}\int_{\Omega} a_{1} |\nabla (v_{1}-v_{2})|^{2}|_{t=0} \,dx \\
&& \qquad = - \int_{0}^{s} \left(\int_{\Omega} (a_{1}-a_{2}) \nabla u_{2}\cdot \nabla(v_{1}-v_{2})\,dx\right)\,dt
        - \frac{1}{2} \int_{0}^{s}\int_{\Omega} a_{1}' |\nabla (v_{1}-v_{2})|^{2} \,dx\,dt \\
\end{eqnarray*}

Denote $C_6= (\varepsilon +  C^2_{1}C^2_{2} ||u_{0}||^2_{H^{1}(\Omega)})^{\frac{p-2}{2}}$ and $C_7= (2-p)N ( C_{1}C_{2})^2 \varepsilon^{\frac{p-4}{2}} ||u_{0}||^2_{H^{1}(\Omega)}$, and by inequalities (\ref{Ineqa2}) and (\ref{Ineqat2}) we get
\begin{eqnarray*}
&& \frac{1}{2}\int_{\Omega} |u_{1}-u_{2}|^{2} \,dx
+ \eta \int_{0}^{s}\int_{\Omega} |u_{1}-u_{2}|^{2} \,dx\,dt
 + \frac{C_6}{2}\int_{\Omega}  |\nabla (v_{1}-v_{2})|^{2}|_{t=0} \,dx \\
&& \leq \int_{0}^{s} \left( ||a_{1}-a_{2}||_{L^{\infty}(\Omega)} \bigg(\int_{\Omega}  |\nabla u_{2}|^{2}\,dx\bigg)^{1/2}
\bigg(\int_{\Omega}  |\nabla (v_{1}-v_{2})|^{2}\,dx\bigg)^{1/2} \right)\,dt + \frac{C_{7}}{2} \int_{0}^{s}\int_{\Omega}  |\nabla (v_{1}-v_{2})|^{2} \,dx\,dt \\
\end{eqnarray*}
Since $G_{\sigma}$ is smooth, there is for every $p\geq1$ a positive constant $C_{p,\sigma}$ depending only on a $p$ and $\sigma$ such that
$$
||a_{1}-a_{2}||_{L^{\infty}(\Omega)} \leq C_{p,\sigma}||u_{1}-u_{2}||_{L^{2}(\Omega)}
$$
We have
\begin{eqnarray*}
&& \frac{1}{2}\int_{\Omega} |u_{1}-u_{2}|^{2} \,dx
+ \eta \int_{0}^{s}\int_{\Omega} |u_{1}-u_{2}|^{2} \,dx\,dt
 + \frac{C_6}{2}\int_{\Omega}  |\nabla (v_{1}-v_{2})|^{2}|_{t=0} \,dx \\
&& \leq \int_{0}^{s} \left( C_{p,\sigma}\bigg(\int_{\Omega}  |u_{1}-u_{2}|^{2}\,dx\bigg)^{1/2} \bigg(\int_{\Omega}  |\nabla u_{2}|^{2}\,dx\bigg)^{1/2}
\bigg(\int_{\Omega}  |\nabla (v_{1}-v_{2})|^{2}\,dx\bigg)^{1/2} \right)\,dt  \\
&& \qquad\qquad + \frac{C_{7}}{2} \int_{0}^{s}\int_{\Omega}  |\nabla (v_{1}-v_{2})|^{2} \,dx\,dt \\
\end{eqnarray*}
Since $u\in H^{1}(\Omega)$, then a positive constant $C_8$ exists such that
\begin{eqnarray*}
&& \frac{1}{2}\int_{\Omega} |u_{1}-u_{2}|^{2} \,dx
+ \eta \int_{0}^{s}\int_{\Omega} |u_{1}-u_{2}|^{2} \,dx\,dt
 + \frac{C_6}{2}\int_{\Omega}  |\nabla (v_{1}-v_{2})|^{2}|_{t=0} \,dx \\
&& \quad \leq \int_{0}^{s} \left( C_{p,\sigma} C_{8} \bigg(\int_{\Omega}  |u_{1}-u_{2}|^{2}\,dx\bigg)^{1/2}
\bigg(\int_{\Omega}  |\nabla (v_{1}-v_{2})|^{2}\,dx\bigg)^{1/2} \right)\,dt + \frac{C_{7}}{2} \int_{0}^{s}\int_{\Omega}  |\nabla (v_{1}-v_{2})|^{2} \,dx\,dt \\
\end{eqnarray*}
Applying Young's inequality, a positive constant $C_{10}$ exists such that
\begin{eqnarray}\label{eq330}
&& \frac{1}{2}\int_{\Omega} |u_{1}-u_{2}|^{2} \,dx
+ \eta \int_{0}^{s}\int_{\Omega} |u_{1}-u_{2}|^{2} \,dx\,dt
 + \frac{C_6}{2} \int_{\Omega}  |\nabla (v_{1}-v_{2})|^{2}|_{t=0} \,dx \nonumber\\
&& \qquad \leq C_{10}\int_{0}^{s} \left( \int_{\Omega}  |u_{1}-u_{2}|^{2}\,dx
+ \int_{\Omega}  |\nabla (v_{1}-v_{2})|^{2}\,dx \right)\,dt
\end{eqnarray}

Define
$$
w_{k}(t)=\int_{0}^{t}u_{k}(\tau)\,d\tau,\quad 0<t<T,\quad k=1,2.
$$
Then (\ref{eq330}) can be rewritten as
\begin{eqnarray}\label{eq335}
&&\frac{1}{2}\int_{\Omega} |u_{1}-u_{2}|^{2} \,dx
+ \eta \int_{0}^{s}\int_{\Omega} |u_{1}-u_{2}|^{2} \,dx\,dt
 + \frac{C_6}{2} \int_{\Omega}  |\nabla (w_{1}-w_{2})(s)|^{2} \,dx \\
&& \qquad \leq C_{10}\int_{0}^{s} \left( \int_{\Omega}  |u_{1}-u_{2}|^{2}\,dx
+ \int_{\Omega}  |\nabla (w_{1}-w_{2})(s) -  \nabla (w_{1}-w_{2})(t) |^{2}\,dx \right)\,dt\nonumber
\end{eqnarray}
Using the inequality
$$
||\nabla (w_{1}-w_{2})(s) -  \nabla (w_{1}-w_{2})(t) ||^{2}_{L^{2}(\Omega)} \leq
2 ||\nabla (w_{1}-w_{2})(s) ||^{2}_{L^{2}(\Omega)} +
2 ||\nabla (w_{1}-w_{2})(t) ||^{2}_{L^{2}(\Omega)}
$$
we have
\begin{eqnarray}\label{eq340}
&& \frac{1}{2}\int_{\Omega} |u_{1}-u_{2}|^{2} \,dx
+ \eta \int_{0}^{s}\int_{\Omega} |u_{1}-u_{2}|^{2} \,dx\,dt
 + \frac{C_6}{2} \int_{\Omega}  |\nabla (w_{1}-w_{2})(s)|^{2} \,dx \\
&& \qquad \leq C_{10}\int_{0}^{s} \left( \int_{\Omega}  |u_{1}-u_{2}|^{2}\,dx
+ 2\int_{\Omega}  |\nabla (w_{1}-w_{2})(t) |^{2}\,dx \right)\,dt + 2sC_{10}  ||\nabla (w_{1}-w_{2})(s) ||^{2}_{L^{2}(\Omega)} \nonumber
\end{eqnarray}

If we choose $T=T_{1}$ sufficiently small and $\varepsilon_{1} >0$ such that $C_6/2 - 2 T_{1}  C_{10} \geq \varepsilon_{1}$, then, for $0\leq s\leq T_{1}$, we have

\begin{eqnarray}\label{eq345}
&&\frac{1}{2}\int_{\Omega} |u_{1}-u_{2}|^{2} \,dx
+ \eta \int_{0}^{s}\int_{\Omega} |u_{1}-u_{2}|^{2} \,dx\,dt
 + \varepsilon_{1} \int_{\Omega}  |\nabla (w_{1}-w_{2})(s)|^{2} \,dx \\
&& \qquad \leq C_{10}\int_{0}^{s} \left( \int_{\Omega}  |u_{1}-u_{2}|^{2}\,dx
+ \int_{\Omega}  |\nabla (w_{1}-w_{2})(t) |^{2}\,dx \right)\,dt, \nonumber
\end{eqnarray}
which implies that
\begin{eqnarray*}
\frac{1}{2}\int_{\Omega} |u_{1}-u_{2}|^{2} \,dx
+ \varepsilon_{1} \int_{\Omega}  |\nabla (w_{1}-w_{2})(s)|^{2} \,dx  \leq C_{10}\int_{0}^{s} \left( \int_{\Omega}  |u_{1}-u_{2}|^{2}\,dx
+ \int_{\Omega}  |\nabla (w_{1}-w_{2})(t) |^{2}\,dx \right)\,dt
\end{eqnarray*}
and finally, if we define $C=(2+1/\varepsilon_{1}) C_{10}$, we obtain
\begin{eqnarray*}
\int_{\Omega} |u_{1}-u_{2}|^{2} \,dx
+  \int_{\Omega}  |\nabla (w_{1}-w_{2})(s)|^{2} \,dx
\leq C\int_{0}^{s} \left( \int_{\Omega}  |u_{1}-u_{2}|^{2}\,dx
+ \int_{\Omega}  |\nabla (w_{1}-w_{2})(t) |^{2}\,dx \right)\,dt
\end{eqnarray*}

Using Gr\"onwall's inequality, we obtain that $u_{1} - u_{2}=0$ on $(0,T_{1}]$.
By applying the argument on the intervals $(T_{1},T_{2}]$, $(T_{2},T_{3}]$, and so on, we see that
$u_{1} - u_{2}=0$ on $(0,T]$.
\end{proof}

\section{Regularization property of the damped flow.}

On the first glimpse, the PDE-based formulation (\ref{SecondOrderFlow}) looks better than the original variational formulation (\ref{energyEu}) because of the absence of the regularization parameter $\alpha$, which is always an obstacle for solving an ill-posed inverse problem. Unfortunately, the ill-posedness remains. Indeed, the terminating time $T$ of the damped flow (\ref{SecondOrderFlow}), instead of $\alpha$ in the original problem (\ref{energyEu}), plays the role of the regularization parameter for the image denoising problem. If the damped flow is discretized, then the formulation (\ref{SecondOrderFlow}) presents a second order iteration scheme; see Section 4 for details. The choice of the terminating time $T$ for the damped flow exactly coincides with the stopping rule for the asymptotical regularization and its generalization, see e.g. (\cite{Tautenhahn1994,ZhangHof2018,ZhangHof2019,GongHofmannZhang2019}).

In this section, we devote our researches to the method of choosing the terminating time $T$. First, let us consider the long-term behavior of the damped flow (\ref{SecondOrderFlow}). Though we only proved the local well-posedness of the dynamical system (\ref{SecondOrderFlow}), we still assume the global existence and uniqueness of the solution to (\ref{SecondOrderFlow}), which will be used for the stability analysis with respect to the noisy image.

Denote by $U^*$ the equilibrium solutions of (\ref{SecondOrderFlow}), namely, $\textmd{div}( (\varepsilon^{\frac{2}{p-2}} + |\nabla G_{\sigma} \star u^*|^{2})^{\frac{p-2}{2}} \nabla u^* ) =0$ for all $u^*\in U^*$. Note that $\textmd{div}\left( (\varepsilon^{\frac{2}{p-2}} + |\nabla G_{\sigma} \star \cdot|^{2})^{\frac{p-2}{2}} \nabla \cdot \right)$ is a monotone operator in a Banach space, therefore, using the Galerkin method, it is not difficult to show that there exists a solution to the equation $\textmd{div}( (\varepsilon^{\frac{2}{p-2}} + |\nabla G_{\sigma} \star u|^{2})^{\frac{p-2}{2}} \nabla u ) =0$ with the zero Neumann boundary condition. Moreover, such a solution is unique up to an overall additive constant. Obviously, $u=const.$ is a solution to the equation $\textmd{div}\left( (\varepsilon^{\frac{2}{p-2}} + |\nabla G_{\sigma} \star u|^{2})^{\frac{p-2}{2}} \nabla u \right)$ $=0$. Hence, we conclude that $U^*=\{u: u'=0\}$. Then, based on the results from~\cite{Haraux1988} and~\cite{Haraux2007} (Theorem 2.1 in~\cite{Haraux1988} and Theorem 3.1 in~\cite{Haraux2007}), we have the convergence result of the global and bounded solutions of problem (\ref{SecondOrderFlow}), i.e., the following theorem holds.
\begin{theorem}\label{LongtimeThm}
Let $\Omega$ be a bounded, open, and connected set in $\mathbf{R}^N$ ($N=2, 3$) having a boundary $\partial\Omega$ of class $C^2$. Then, there exists a constant $u_\infty$ such that the solution $u$ to the equation (\ref{SecondOrderFlow}) satisfies
\begin{equation}\label{decay}
\|u_t\|_{L^2(\Omega)} + \|u(x,t)-u_\infty\|_{L^2(\Omega)} \leq C(u_0,\Omega) e^{-c_d t},
\end{equation}
where $C(u_0,\Omega)>0$ depends on the initial data $u_0$ and the geometry of domain $\Omega$, while $c_d=c_d(\eta,\varepsilon,\sigma,p,\Omega)>0$ depends on model parameters $\eta,\varepsilon,\sigma,p,\Omega$, but not on $u_0$ and $t$.
\end{theorem}

Suppose that instead of the exact image $\bar{u}$ we are given approximate one, $u^\delta_0$, such that $\frac{\|u^\delta_0-\bar{u}\|_{L^2(\Omega)}}{\|\bar{u}\|_{L^2(\Omega)}} \leq \delta$, where the positive number $\delta$ denotes the degree of difference between the accurate image $\bar{u}$ and polluted image $u^\delta_0$. Obviously, if there is no noise, i.e. $\delta=0$, no denoising algorithm is needed. In this case, $T(\delta)=0$.

\begin{theorem}
(A priori selection method for $T(\delta)$)\\ Denote by $u(x,t)$ the solution of the damped flow (\ref{SecondOrderFlow}) with the initial data $u^\delta_0(x)$. Then, if the terminating time point is chosen as $T(\delta)=C_1 \ln(1+ C_2 \delta^\gamma)$, where $C_1,C_2,\gamma$ are positive constants independent of $\delta$, the approximate solution $u(T(\delta),x)$ converges to the exact image $\bar{u}(x)$ as $\delta\to0$.
\end{theorem}

\begin{proof}
Using the estimate (\ref{decay}), we obtain
\begin{eqnarray*}\label{Ineq}
\begin{array}{ll}
& \| u(T,x) - \bar{u}(x)\|_{L^2(\Omega)}  \leq \| u(T(\delta),x) - u^\delta_0(x)\|_{L^2(\Omega)} + \| u^\delta_0(x) - \bar{u}(x)\|_{L^2(\Omega)} \leq \int^T_0 \| u_t\|_{L^2(\Omega)} dt +  \delta \|\bar{u}\|_{L^2(\Omega)} \\ & \qquad\qquad
 \leq \int^T_0 C(u_0,\Omega) e^{-c_d t} dt + \delta \|\bar{u}\|_{L^2(\Omega)} = \frac{C(u_0,\Omega)}{c_d} \left(1- e^{-c_d T} \right) + \delta \|\bar{u}\|_{L^2(\Omega)}.
\end{array}
\end{eqnarray*}

Note that the terminating time point is chosen as $T(\delta)=C_1 \ln(1+ C_2 \delta^\gamma)$. By combining the above inequalities, we can deduce that
\begin{equation}\label{Ineq1new}
\| u(T(\delta),x) - \bar{u}(x)\|_{L^2(\Omega)} \leq \frac{C(u_0,\Omega)}{c_d} \left( 1- \left(1+ C_2 \delta^\gamma \right)^{-c_d C_1} \right) + \delta \|\bar{u}\|_{L^2(\Omega)} .
\end{equation}

On the other hand, for a sufficiently small $\delta$, the inequality $\left(1+ C_2 \delta^\gamma \right)^{-c_d C_1} \geq 1 - c_d C_1C_2 \delta^\gamma$ holds. Therefore, by (\ref{Ineq1new}), we can deduce that
\begin{eqnarray*}
\| u(T(\delta),x) - \bar{u}(x)\|_{L^2(\Omega)} \leq C(u_0,\Omega) C_1C_2 \delta^\gamma  + \|\bar{u}\|_{L^2(\Omega)} \delta,
\end{eqnarray*}
which implies the convergence of the obtained approximate solution $u(T(\delta),x)$.
\end{proof}

By the proof of the above theorem, we know that under the a priori selection method for the final time point $T(\delta)=C_1 \ln(1+ C_2 \delta^\gamma)$, the convergence rate of the method is $\min(\gamma,1)$. However, an a priori parameter choice is not suitable in practice, since a good terminating time point $T$ requires knowledge of the unknown image $\bar{u}(x)$. Moreover, there are intractable factors, $C_1, C_2$, around the parameter. This knowledge is not necessary for a posteriori parameter choice. Here, we develop a modified Morozov's discrepancy principle of choosing the terminating time point $T$.

Define by
\begin{eqnarray*}
\sigma(T)=\| u(T,x) - u^\delta_0(x)\|_{L^2(\Omega)} / \|u^\delta_0(x)\|_{L^2(\Omega)}
\end{eqnarray*}
the tolerability ratio of the difference between the estimated and noisy images.

Introduce the discrepancy function
\begin{eqnarray*}
\chi(T)=\sigma(T) - \delta,
\end{eqnarray*}
which describes the difference between the tolerability ratio of the denoised image and the degree of the measured noisy image. Obviously, by Theorem 2.3, $\chi(T)$ is a continuous function.

\begin{theorem}
(A posteriori selection method for $T(\delta)$)\\
Suppose that the noisy image $u^\delta_0$ is not an ``almost-constant'', i.e.
\begin{equation}\label{assumptionData}
\inf\limits_{c \textrm{~is a constant}} \|c-u^\delta_0(x)\| = \mu>0.
\end{equation}
Then, there exists a positive number $\delta_0>0$ such that for all $\delta\in (0,\delta_0]$, the discrepancy function $\chi(T)$ admits at least one positive root. Moreover, the approximate solution $u(T(\delta),x)$, with the terminating time point chosen as the positive root of $\chi(T)$, converges to the exact image $\bar{u}(x)$ as $\delta\to0$.
\end{theorem}

\begin{proof}
Combine the estimate (\ref{decay}) and the assumption of $u^\delta_0(x)$ in (\ref{assumptionData}), and one can deduce that for $\delta<\mu/\|u^\delta_0(x)\|$:
\begin{eqnarray*}
\lim_{T\to\infty} \chi(T) \geq \frac{\mu}{\|u^\delta_0(x)\|} - \delta>0.
\end{eqnarray*}
On the other hand, by the definition of the approximate solution $u$ (the solution to (\ref{SecondOrderFlow})), we have $\chi(0)=-\delta<0$. Since $\chi(T)$ is a continuous function, $\chi(T)$ must admit at least one positive root.

Now, consider the convergence property of the solution $u(T(\delta),x)$. By the definition of the final time point in selection (the positive root of $\chi(T)$), we obtain
\begin{eqnarray*}
\| u(T,x) - \bar{u}(x)\| \leq \| u(T(\delta),x) - u^\delta_0(x)\| + \| u^\delta_0(x) - \bar{u}(x)\| \leq (\|u^\delta_0(x)\|+1) \delta,
\end{eqnarray*}
which implies the convergence of the desired approximate solution $u(T,x)$ immediately.
\end{proof}

\begin{remark}
If function $\chi(T)$ has more than one positive root, then, any of root $T$ gives a stable approximate image $u(T,x)$. In practice, one can choose $T^*=\inf_{\chi(T)=0} T$, i.e. $\sigma(T)<\delta$ for all $T<T^*$ and $\sigma(T^*)=\delta$. In other words, $T^*$ is the first time point for which the tolerability ratio $\sigma(T)$ coincides with the data error.
\end{remark}

\section{A discrete damped flow.}

Loosely speaking, the damped flow (\ref{SecondOrderFlow}) with an appropriate numerical discretization yields a discrete second order regularization method. Just like the Runge-Kutta integrators~\cite{Rieder-2005} or the exponential integrators~\cite{Hochbruck-1998} for solving first order equations, the damped symplectic integrators are extremely attractive for solving second order equations (\ref{SecondOrderFlow}), since the schemes are closely related to the canonical transformations~\cite{Hairer-2006}, and the trajectory of the discretized second flows are usually more stable. In this section, based on the St\"{o}rmer-Verlet method, we develop a discrete damped flow for image denoising.

For simplicity and clarity of statements, let $\Omega$ denote a rectangle region in $\mathbf{R}^2$, and let us consider a uniform grid $\Omega_{MN}=\{(x_i,y_j)\}^{M,N}_{i,j=1}$ in $\Omega$ with the uniform step size $h=x_{i+1}-x_i=y_{j+1}-y_j$. Define $\mathbf{u}(t)=[u(x_i,y_j,t)]^{M,N}_{i,j=1}$. Denote $\mathbf{u}^{k}$ as the projection of $u(x,y,t)$ at the spacial grid $\Omega_{MN}$ and time point $t=t_k$. We approximate the $\textmd{div}\left( a^{\varepsilon} (u) \nabla u \right)$ by a linear one~-- $\textmd{div}\left( a^{\varepsilon}(\mathbf{u}^{k-1}) \nabla \mathbf{u}^{k} \right)$, where $a^{\varepsilon}(u)$ is defined in (\ref{a}). Using the central difference discretization rule, we have
\begin{equation}\label{linearApproximation0}
\begin{array}{rl}
& \textmd{div}\left( a(\mathbf{u}^{\varepsilon, k-1}) \nabla \mathbf{u}^{k} \right) = D_{x,\frac{h}{2}} \left( a^{\varepsilon, k-1}_{i,j} D_{x,\frac{h}{2}}  \mathbf{u}^{k}_{i,j} \right) + D_{y,\frac{h}{2}} \left( a^{\varepsilon, k-1}_{i,j} D_{y,\frac{h}{2}}  \mathbf{u}^{k}_{i,j} \right)  \\ & =
D_{x,\frac{h}{2}} \left( a^{\varepsilon, k-1}_{i,j} \frac{\mathbf{u}^{k}_{i+\frac{1}{2},j} - \mathbf{u}^{k}_{i-\frac{1}{2},j}}{h} \right) + D_{y,\frac{h}{2}} \left( a^{\varepsilon, k-1}_{i,j} \frac{\mathbf{u}^{k}_{i,j+\frac{1}{2}} - \mathbf{u}^{k}_{i,j-\frac{1}{2}}}{h} \right)  \\ & = \frac{1}{h^2} \Big\{ a^{\varepsilon, k-1}_{i-\frac{1}{2},j} \mathbf{u}^{k}_{i-1,j} +  a^{\varepsilon, k-1}_{i,j-\frac{1}{2}} \mathbf{u}^{k}_{i,j-1} - \left( a^{\varepsilon, k-1}_{i-\frac{1}{2},j} + a^{\varepsilon, k-1}_{i,j-\frac{1}{2}} + a^{\varepsilon, k-1}_{i+\frac{1}{2},j} + a^{\varepsilon, k-1}_{i,j+\frac{1}{2}} \right) \mathbf{u}^{k}_{i,j} + a^{\varepsilon, k-1}_{i,j+\frac{1}{2}} \mathbf{u}^{k}_{i,j+1} +  a^{\varepsilon, k-1}_{i+\frac{1}{2},j} \mathbf{u}^{k}_{i+1,j} \Big\} ,
\end{array}
\end{equation}
where
\begin{equation}\label{Linear_a}
a^{\varepsilon, k-1}_{i-\frac{1}{2},j} = \left( \varepsilon + |\nabla \mathbf{G}_{\sigma}\star \mathbf{u}^{k-1}_{i-\frac{1}{2},j} |^{2} \right)^{\frac{p-2}{2}} .
\end{equation}
Here we use $\mathbf{u}^{k-1}_{i-\frac{1}{2},j} = \frac{\mathbf{u}^{k-1}_{i-1,j} + \mathbf{u}^{k-1}_{i,j}}{2}$ to approximate $\mathbf{u}^{k-1}_{i-\frac{1}{2},j}$ in $a^{\varepsilon, k-1}_{i-\frac{1}{2},j}$ and
$\nabla \mathbf{G}_{\sigma}$ is the project of function $\nabla G_{\sigma}$ on the same grid $\Omega_{MN}$.

\begin{definition}
Given a matrix $\mathbf{u}\in \mathbf{R}^{M}\times \mathbf{R}^{N}$, one can obtain a vector $\vec{\mathbf{u}}\in \mathbf{R}^{MN}$ by stacking the columns of $\mathbf{u}$. This defines a linear operator $vec: \mathbf{R}^{M}\times \mathbf{R}^{N} \to \mathbf{R}^{MN}$,
\begin{eqnarray*}
vec(\mathbf{u}) = (\mathbf{u}_{1,1}, \mathbf{u}_{2,1}, \cdot\cdot\cdot, \mathbf{u}_{M,1}, \mathbf{u}_{1,2}, \mathbf{u}_{2,2}, \cdot\cdot\cdot, \mathbf{u}_{M,1}, \cdot\cdot\cdot, \mathbf{u}_{1,N}, \mathbf{u}_{2,N}, \cdot\cdot\cdot, \mathbf{u}_{M,N})^T , ~
\vec{\mathbf{u}} = vec(\mathbf{u}), ~ \vec{\mathbf{u}}_q = \mathbf{u}_{i,j},
\end{eqnarray*}
where $q=(i-1)M+j$.
This corresponds to a lexicographical column ordering of the components in the matrix $\mathbf{u}$. The symbol $array$ denotes the inverse of the $vec$ operator. That means
\begin{eqnarray*}\label{array}
array(vec(\mathbf{u})) = \mathbf{u}, \quad vec(array(\vec{\mathbf{u}})) = \vec{\mathbf{u}},
\end{eqnarray*}
whenever $\mathbf{u}\in \mathbf{R}^{M}\times \mathbf{R}^{N}$ and $\vec{\mathbf{u}}\in \mathbf{R}^{MN}$.
\end{definition}

Based on the above definition, rewrite (\ref{linearApproximation0}) as the matrix form, $\mathbf{F}^{k-1} \vec{\mathbf{u}}^{k}$, where the matrix $\mathbf{F}^{k-1}$ is dependent only on $\vec{\mathbf{u}}^{k-1}$.

\begin{proposition}\label{NegativeEigen}
All eigenvalues of $\mathbf{F}^{k}$ ($k=2, \cdots$) are non-positive.
\end{proposition}

\begin{proof}
By the definition of $\mathbf{F}^{k}$, it is not difficult to show that $\mathbf{F}^{k}$ is a symmetriccal and diagonally dominant matrix. Then, all eigenvalues of $\mathbf{F}^{k}$ ($k=2, \cdots$) are real and, by Gershgorin's circle theorem, for each eigenvalue $\lambda$ an index $\nu$ exists such that:
\begin{eqnarray*}
\lambda \in \left[ [\mathbf{F}^{k}]_{\nu,\nu} - \sum^{MN}_{\imath \neq \nu} |[\mathbf{F}^{k}]_{\nu,\imath}|, [\mathbf{F}^{k}]_{\nu,\nu} + \sum^{MN}_{\imath \neq \nu} |[\mathbf{F}^{k}]_{\nu,\imath}| \right],
\end{eqnarray*}
which implies, by definition of the diagonal dominance, $\lambda\leq 0$. Here, $[\mathbf{F}^{k}]_{\nu,\imath}$ denotes the element of the matrix $\mathbf{F}^{k}$ at the position $(\nu,\imath)$.
\end{proof}

Denote $\vec{\mathbf{v}}^k=\frac{d \vec{\mathbf{u}}^k}{dt}$. In this work, the St\"{o}rmer-Verlet method is employed to solve PDE (\ref{SecondOrderFlow}), namely
\begin{equation}\label{symplectic}
\left\{\begin{array}{l}
\vec{\mathbf{v}}^{k+\frac{1}{2}} = \vec{\mathbf{v}}^{k} + \frac{\Delta t_k}{2} \left( \mathbf{F}^{k-1} \vec{\mathbf{u}}^{k} - \eta \vec{\mathbf{v}}^{k+\frac{1}{2}} \right), \\
\vec{\mathbf{u}}^{k+1} = \vec{\mathbf{u}}^{k} + \Delta t_k \vec{\mathbf{v}}^{k+\frac{1}{2}}, \\
\vec{\mathbf{v}}^{k+1} = \vec{\mathbf{v}}^{k+\frac{1}{2}} + \frac{\Delta t_k}{2} \left( \mathbf{F}^{k} \vec{\mathbf{u}}^{k+1} - \eta \vec{\mathbf{v}}^{k+\frac{1}{2}} \right), \\
\vec{\mathbf{u}}_{0}=\vec{\mathbf{u}}^\delta_0, \vec{\mathbf{v}}_0=0,
\end{array}\right.
\end{equation}
where $\vec{\mathbf{u}}^\delta_0= vec(\mathbf{u}^\delta_0)$ and $\mathbf{u}^\delta_0$ is the project of $u^\delta_0(x)$ on the grid $\Omega_{MN}$.

Now, we are in a position to give a numerical analysis for the scheme (\ref{symplectic}).

Denote by $\mathbf{z}^k=(\vec{\mathbf{u}}^{k};\vec{\mathbf{v}}^k)$, and $\mathbf{E}$ the identity matrix of size $MN$, then, equation (\ref{symplectic}) can be rewritten as
\begin{equation}\label{ODE2}
\mathbf{z}^{k+1} = \mathbf{B}^{k} \mathbf{A}^{k-1} \mathbf{z}^k,
\end{equation}
where
\begin{equation}\label{B1}
\mathbf{A}^{k-1} = \frac{2}{2+ \eta \Delta t_k} \left( \begin{array}{cc}
\left(1+\frac{\eta \Delta t_k}{2} \right) \mathbf{E} + \frac{\Delta t_k^2}{2} \mathbf{F}^{k-1}  & \Delta t_k \mathbf{E} \\
\frac{\Delta t_k}{2} \mathbf{F}^{k-1}  & \mathbf{E}
\end{array}\right),~ \mathbf{B}^{k} = \left( \begin{array}{cc}
\mathbf{E} & \mathbf{0} \\
\frac{\Delta t_k}{2} \mathbf{F}^{k}  & \left(1-\frac{\eta \Delta t_k}{2} \right) \mathbf{E} + \frac{\Delta t_k^2}{2} \mathbf{F}^{k}
\end{array}\right).
\end{equation}

\begin{theorem}
(Boundedness) If
\begin{equation}\label{parametersDt}
\Delta t_k \leq \min\left\{ \frac{\eta}{\sqrt{\lambda^{(k)}_{max}}}, \sqrt{\frac{8}{\lambda^{(k)}_{max}} + \left(\frac{\eta}{\lambda^{(k)}_{max}} \right)^2 }  - \frac{\eta}{\lambda^{(k)}_{max}} \right\}
\end{equation}
then, the scheme (\ref{symplectic}) is uniformly bounded.
\end{theorem}

\begin{proof}
By Proposition \ref{NegativeEigen}, all the eigenvalues of $\mathbf{F}^{k}$ are non-positive. By noting that $\mathbf{F}^{k}$ is a symmetrical matrix, there exists a decomposition $\mathbf{F}^{k}=\Phi^{k} \Lambda^k (\Phi^{k})^T$, where $\Phi^{k}$ is an unitary matrix and $\Lambda^k=-\textmd{diag}(\lambda^{(k)}_i)$, where $\lambda^{(k)}_i\geq0$, $i=1, \cdots, MN$.

It is well known that, a sufficient condition for the boundedness of a dynamical system is $\|\mathbf{B}^{k} \mathbf{A}^{k-1}\|_2 \leq1$, i.e. the composite mapping $\mathbf{B}^{k} \mathbf{A}^{k-1}$ is non-expansive. By the directly calculation, the eigenvalues of matrices $\mathbf{B}^{k}$ are
\begin{eqnarray*}
\nu^{(k)}_i(\mathbf{B}^{k})= 1- \frac{\eta \Delta t_k}{2} - \frac{\Delta t^2_k}{2} \lambda^{(k)}_i \textmd{~for~} i=1, \cdots, MN, \quad \textmd{~and~} \nu^{(k)}_i\equiv1 \textmd{~for~} i=1+MN, \cdots, 2MN,
\end{eqnarray*}
which implies that $|\nu^{(k)}_i(\mathbf{B}^{k})|\leq1$ for all $i,k$ by estimate (\ref{parametersDt}). Therefore, using the relation $\|\mathbf{B}^{k} \mathbf{A}^{k-1}\|_2 \leq \|\mathbf{B}^{k}\|_2 \|\mathbf{A}^{k-1}\|_2= \|\mathbf{A}^{k-1}\|_2$ it is sufficient to show that for the given time step size $\Delta t_k$ in (\ref{parametersDt}), the corresponding eigenvalues of $\mathbf{A}^{k-1}$ are not greater than the unit.

The eigenvalues of matrices $\mathbf{A}^{k-1}$ are
\begin{eqnarray*}
\mu^{(k-1)}_{i,\pm} (\mathbf{A}^{k-1})= \frac{2}{2+ \eta \Delta t_k} \left(  1 + \frac{\Delta t_k}{4} \left( \eta - \Delta t_k \lambda^{(k-1)}_i \pm \sqrt{ \left( \eta - \Delta t_k \lambda^{(k-1)}_i \right)^2 - 8 \lambda^{(k-1)}_i} \right) \right).
\end{eqnarray*}

Now, we have to show that for all $k$: $|\mu^{(k-1)}_{max} (\eta, \Delta t_k^{(k-1)})| \leq 1$ for the parameter $\Delta t_k^{(k-1)}$ defined by (\ref{parametersDt}).

For simplicity, we ignore the superscript $^{(k-1)}$ from now on. Denote by $i_*$ the index of $\lambda_{i_*}$, corresponding the maximal absolute value of $\mu^{(k-1)}_{i,\pm} (\mathbf{A}^{k-1})$, i.e.
\begin{eqnarray*}
|\mu_{max}| = \frac{2}{2+ \eta \Delta t_k} \max_{+,-} \left| 1 + \frac{\Delta t_k}{4} \left( \eta - \Delta t_k \lambda_{i_*} \pm \sqrt{ \left( \eta - \Delta t_k \lambda_{i_*} \right)^2 - 8 \lambda_{i_*}} \right) \right|.
\end{eqnarray*}

If $\lambda_{i_*}=0$, the theorem holds, obviously, since $|\mu_{max}|\equiv1$ in this case.

Now, consider the case when $\lambda_{i_*}>0$. There are three possible cases here: the overdamped case ($ \left( \eta - \Delta t_k \lambda_{i_*} \right)^2 > 8 \lambda_{i_*}$), the underdamped case ($ \left( \eta - \Delta t_k \lambda_{i_*} \right)^2 < 8 \lambda_{i_*}$), and the critical damped case ($ \left( \eta - \Delta t_k \lambda_{i_*} \right)^2 = 8 \lambda_{i_*}$). Let us consider these cases respectively.

For the chosen time step size $\Delta t_k$ in (\ref{parametersDt}), we have $\eta - \Delta t_k \lambda_{i_*}\geq0$. Therefore, for the overdamped case,
\begin{eqnarray*}
|\mu_{max}| = \frac{2}{2+ \eta \Delta t_k} \left( 1 + \frac{\Delta t_k}{4} \left( \eta - \Delta t_k \lambda_{i_*} + \sqrt{ \left( \eta - \Delta t_k \lambda_{i_*} \right)^2 - 8 \lambda_{i_*}} \right) \right)
\end{eqnarray*}

Define $\eta - \Delta t_k \lambda_{i_*}=a \sqrt{8\lambda_{i_*}}$ ($a>1$), and we have
\begin{eqnarray*}
|\mu_{max}| = \frac{1 + \frac{\Delta t_k}{4} (a + \sqrt{a^2-1}) \sqrt{8\lambda_{i_*}}}{1 + \frac{\Delta t_k}{2} \eta}.
\end{eqnarray*}

Substituting $\eta = \Delta t_k \lambda_{i_*}+a \sqrt{8\lambda_{i_*}}$ in the above equation, we can deduce that
\begin{eqnarray*}
|\mu_{max}| = \frac{1 + \frac{\Delta t_k}{4} (a + \sqrt{a^2-1}) \sqrt{8\lambda_{i_*}}}{1 + \frac{\Delta t_k}{2} (\Delta t_k \lambda_{i_*}+a \sqrt{8\lambda_{i_*}})}
\leq \frac{1 + \frac{\Delta t_k}{2} a \sqrt{8\lambda_{i_*}}}{1 + \frac{\Delta t_k}{2} (\Delta t_k \lambda_{i_*}+a \sqrt{8\lambda_{i_*}})} <1.
\end{eqnarray*}

Now, consider the underdamped case. The complex eigenvalue $\mu_{max}$ satisfies
\begin{eqnarray*}
|\mu_{max}|^2  = \frac{1 + \frac{\Delta t_k}{2} (\eta - \Delta t_k \lambda_{i_*}) + (\frac{\Delta t_k}{4} )^2 8\lambda_{i_*}}{(1 + \frac{\Delta t_k}{2} \eta)^2}.
\end{eqnarray*}
Similarly, if we define $\eta - \Delta t_k \lambda_{i_*}=a \sqrt{8\lambda_{i_*}}$ with $a<1$, we have
\begin{eqnarray*}
|\mu_{max}|^2  = \frac{1 + \frac{\Delta t_k}{2} a \sqrt{8\lambda_{i_*}} + (\frac{\Delta t_k}{4})^2 8\lambda_{i_*}}{(1 + \frac{\Delta t_k}{2} (\Delta t_k \lambda_{i_*}+a \sqrt{8\lambda_{i_*}}))^2}
=  \frac{1 + a \Delta t_k \sqrt{2\lambda_{i_*}} + \frac{\Delta t_k^2}{2} \lambda_{i_*}}{1+\frac{\Delta t_k^4}{4}\lambda^2_{i_*} + (2 a^2 + 1 ) \Delta t_k^2 \lambda_{i_*} + 2a\Delta t_k\sqrt{2\lambda_{i_*}} + a \Delta t_k^3 \lambda_{i_*} \sqrt{2\lambda_{i_*}}} <1
\end{eqnarray*}

Finally, consider the critical damped case. In this case,
\begin{eqnarray*}
|\mu_{max}| = \frac{1 + \frac{\Delta t_k}{4} \sqrt{8\lambda_{i_*}}}{1 + \frac{\Delta t_k}{2} \eta} = \frac{1 + \frac{\Delta t_k}{4} \sqrt{8\lambda_{i_*}}}{1 + \frac{\Delta t_k}{2} (\Delta t_k \lambda_{i_*} + \sqrt{8\lambda_{i_*}} )}<1,
\end{eqnarray*}
which completes the proof.
\end{proof}

By Taylor's theorem and the finite difference formula, it is not difficult to show the consistency of the scheme (\ref{symplectic}). It is well known that boundedness implies the convergence of consistent schemes for any (especially nonlinear) problem, namely, the following theorem holds~\cite{Tadmor}.

\begin{theorem}
(Convergence) The scheme (\ref{symplectic}) is convergent if the time step size is chosen by the criterion (\ref{parametersDt}).
\end{theorem}

\section{The SV-DDF algorithm .}
In this section, we propose an algorithm for image denoising. Various stopping criteria exist for an iteration algorithm (\cite{Gonzalez2007,Scherzer2009,Khanian2014}). In principle, the stopping criterion for image denoising problems should be proposed case by case. In real world problems, in order to obtain a high qualified denoised image, a manual stopping criterion is always required, especially for the PDE-based denoising technique. Nevertheless, an automatic stopping criterion can definitely help people to select a good initial guess of the denoised image.

In this paper, we adapt a frequency domain threshold method based on the fact that noise is usually represented by high frequencies in the frequency domain. To this end, define the high frequencies energy by
\begin{eqnarray*}
\Delta_{N_0} (\mathbf{u}) = \sum_{i+j\geq N_0} \left| \mathcal{F}(\mathbf{u}) (i,j) \right|^2,
\end{eqnarray*}
where $\mathcal{F}(\mathbf{u})$ denotes a 2D discrete Fourier transform of an image $\mathbf{u}$, and $N_0$ presents the high frequencies index. In the simulation, we set $N_0=\lfloor 0.6N^2 \rfloor$, where $\lfloor \cdot \rfloor$ denotes the floor function. Define by
\begin{eqnarray*}
RDE(k)=|\Delta_{N_0} (\mathbf{u}^k) - \Delta_{N_0} (\mathbf{u}^{k-1})| /  \Delta_{N_0} (\mathbf{u}^{k-1}).
\end{eqnarray*}
the relative denoising efficiency. Then, the value of $RDE$ at every iteration can be used as a stopping criterion. Based on this stopping criterion, an algorithm of SV-DDF for image denoising is proposed in Algorithm 1.

\begin{algorithm}[!htb]
\label{Alorithm1}
\caption{The SV-DDF for image denoising.}
\begin{algorithmic}[1]
\Require  Observed noisy image $u^\delta_0$. Parameters $\eta$ and $p$. Tolerance $\varepsilon_0$.

\Ensure A denoised image $\hat{u} \gets array(\vec{\mathbf{u}}^{k})$.

\State $\vec{\mathbf{u}}_{0} \gets vec(u_0)$, $\vec{\mathbf{v}}_{0} \gets 0$, $\Delta t_0 \gets \lambda_{max}(\mathbf{F}^0)$, $\mathbf{F}^{-1} \gets \mathbf{F}^0$, $RDE(0) \gets 1$, $k \gets 0$

\While{$RDE(k)>\varepsilon_0$}

\State $\vec{\mathbf{v}}^{k+\frac{1}{2}} \gets \left( 1+ \frac{\Delta t_k}{2} \eta \right)^{-1} \cdot \left( \vec{\mathbf{v}}^{k} + \frac{\Delta t_k}{2} \mathbf{F}^{k-1} \vec{\mathbf{u}}^{k} \right)$

\State $\vec{\mathbf{u}}^{k+1} \gets \vec{\mathbf{u}}^{k} + \Delta t_k \vec{\mathbf{v}}_{k+\frac{1}{2}}$

\State $\vec{\mathbf{v}}^{k+1} \gets \vec{\mathbf{v}}^{k+\frac{1}{2}} + \frac{\Delta t_k}{2} \left( \mathbf{F}^{k} \vec{\mathbf{u}}^{k+1} - \eta \vec{\mathbf{v}}^{k+\frac{1}{2}} \right)$

\State $k \gets k+1$

\State $RDE(k) \gets |\Delta_{N_0} (\mathbf{u}^k) - \Delta_{N_0} (\mathbf{u}^{k-1})| /  \Delta_{N_0} (\mathbf{u}^{k-1})$

\State $\Delta t_k \gets \lambda_{max}(\mathbf{F}^k)$

\EndWhile
\end{algorithmic}
\end{algorithm}

\section{Numerical experiments.}

In this section, several numerical examples are given to show the feasibility and efficiency of our proposed image denoising approach~-- SV-DDF (Algorithm 1).

Let $\bar{u}$ be the noise-free image, see Fig. \ref{orig}. In this paper, we consider two type of noise structure: (i) The uniformly distributed noise with noise level $\delta$ (cf. (a) in Fig. \ref{NoisyImages}). The noisy image is defined as $u^\delta_0(x) =\max \{0,[1+\delta\cdot(2\,\textrm{rand}(x)-1)]\, \bar{u}(x) \}$, $x\in\Omega$, where ``rand'' returns a pseudo-random value drawn from a uniform distribution on $[0, 1]$. (ii) The salt and pepper dominated noise (cf. (b) in Fig. \ref{NoisyImages}). It should be noted that for images with purely salt and pepper noise, specific approaches, such as median filtering methods (\cite{Lim1990,Juhola1991,Eng2001}), etc., work better than our proposed SV-DDF.

\begin{figure}[!htb]
\centering
\subfigure[]{
\includegraphics[clip, trim=2.5in 2.5in 2.5in 0.5in, width=2.5in]{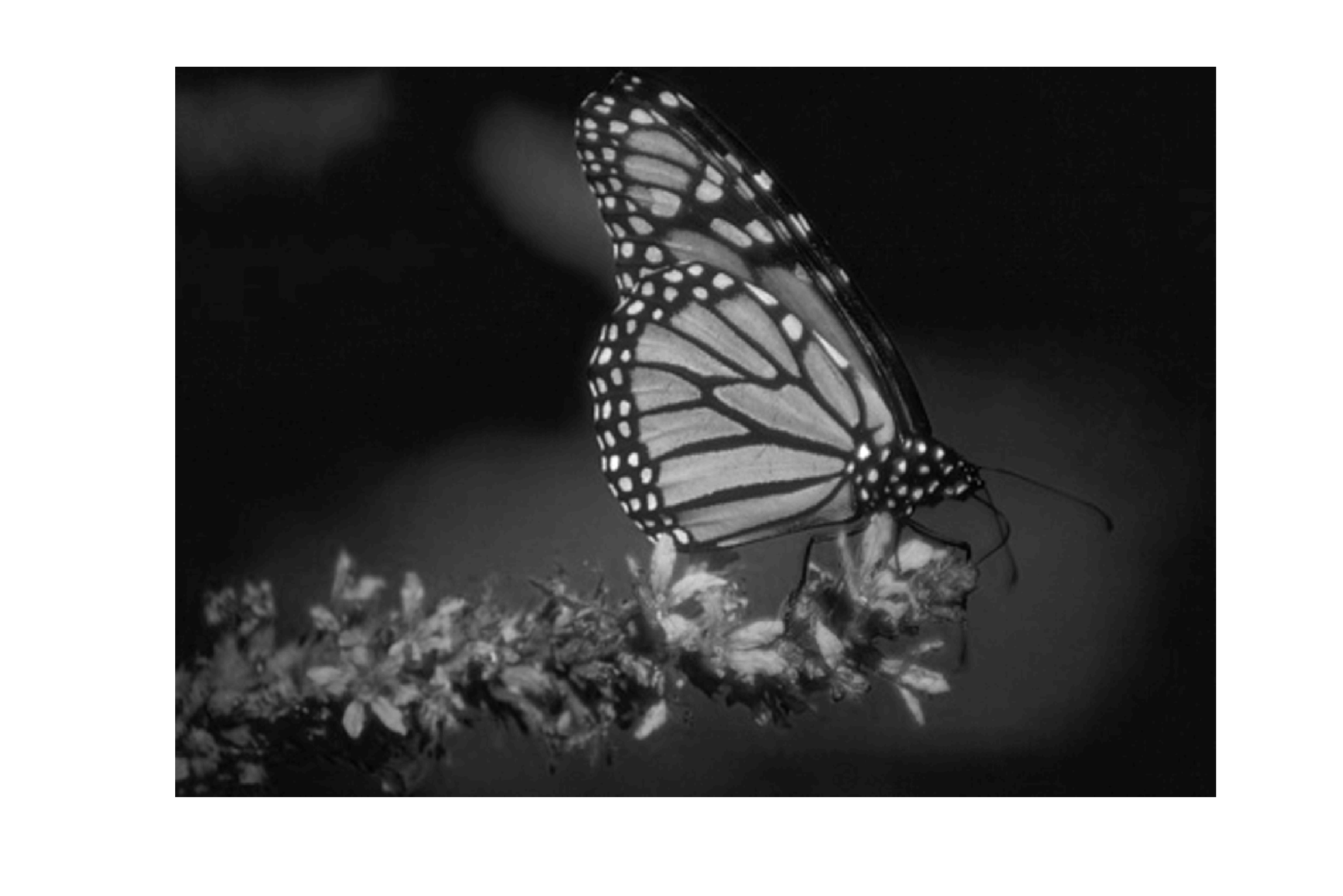}}
\subfigure[]{
\includegraphics[clip, trim=0.5in 0.5in 0in 0in, width=3.1in]{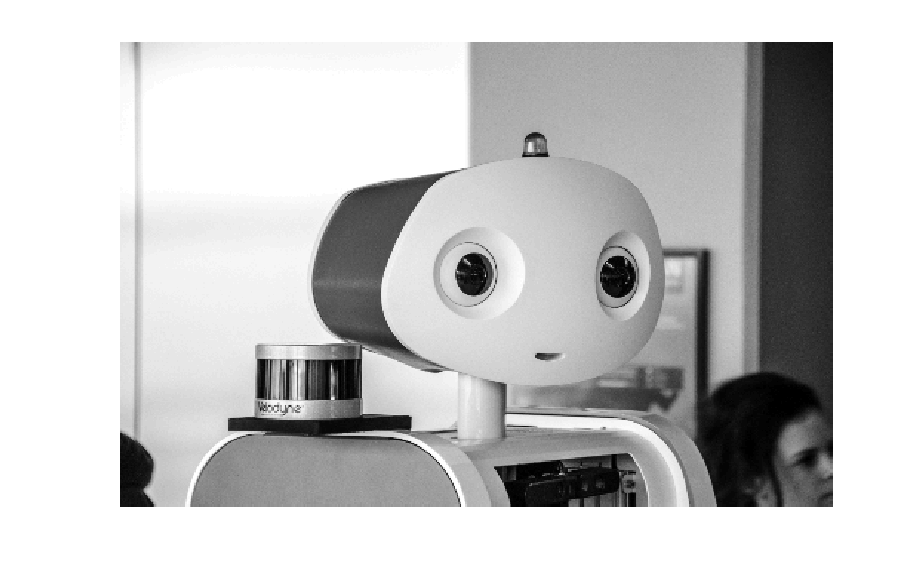}}
\caption{The two noise free images.}
\label{orig}
\end{figure}

\begin{figure}[!htb]
\centering
\subfigure[]{
\includegraphics[clip, trim=2.5in 2.5in 2.5in 0.5in, width=2.5in]{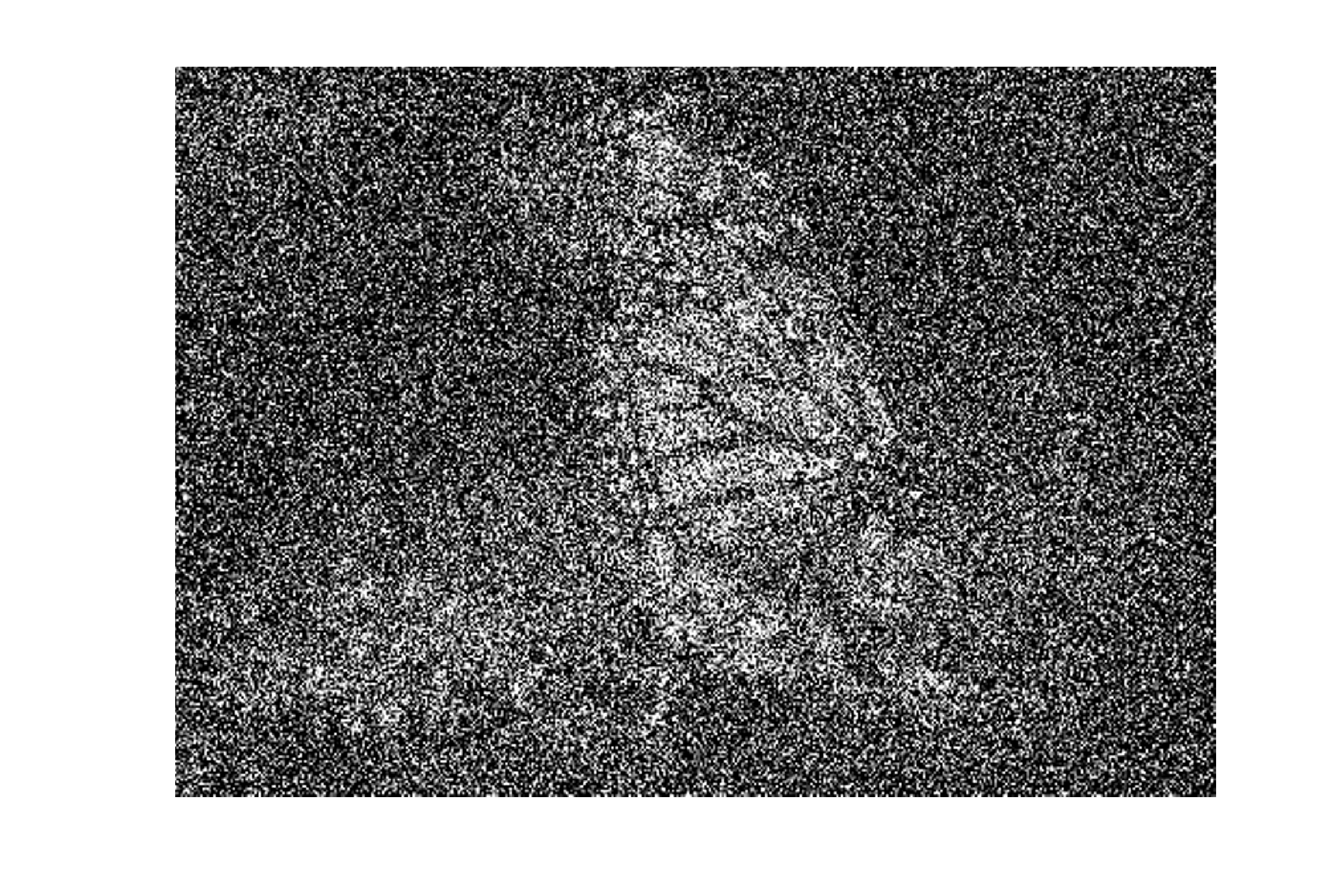}}
\subfigure[]{
\includegraphics[clip, trim=0.5in 0.5in 0in 0in, width=3.1in]{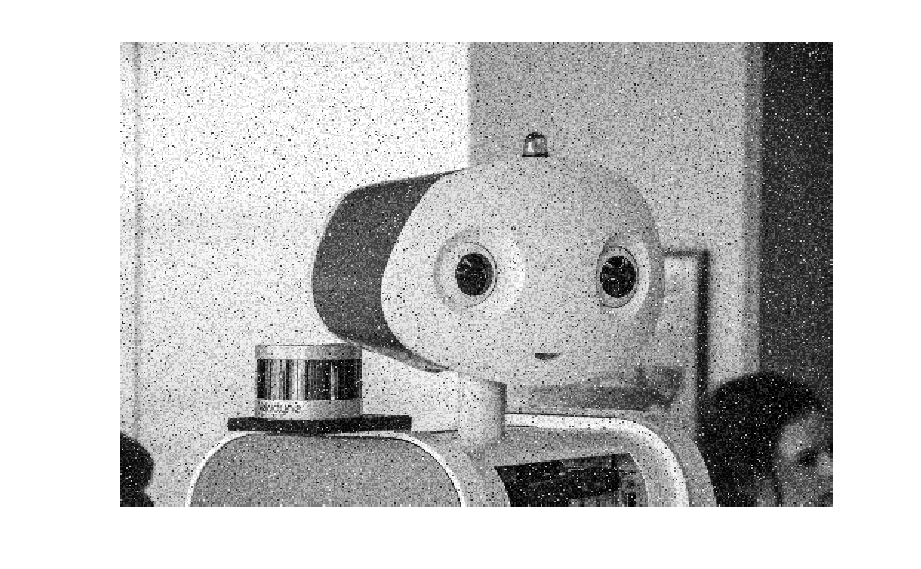}}
\caption{(a) Noisy image with uniformly distributed noise of level $\delta=54\%$. (b) Noisy image with salt and pepper dominated noise.}
\label{NoisyImages}
\end{figure}

To assess the accuracy of the denoised images we use the well known structural similarity error measure SSIM and the Peak Signal-to-Noise Ratio (PSNR) to obtain a quantitative estimate of the denoising performance (\cite{Wang-2004}).

\subsection{Influence of parameters}
The purpose of this paragraph is to explore the dependence of the accuracy of the denoised image with respect to the damping parameter $\eta$ and the method parameter $p$ in order to find proper values in practice. In all simulations below, we set $\epsilon=\sigma=0.001$. Numerical experiments indicate that small changes (less than $1\%$) of the values of $\epsilon$ and $\sigma$ do not significantly influence the output of our method. In Tab. \ref{TabDamping}, we display the results by using different values of the damping parameter $\eta$ and the method parameter $p$ for the two noisy test images in Fig. \ref{NoisyImages}. The results for the first test picture show that we obtain the best result for  $p=1$, $\eta=300$ for both evaluation criteria of SSIM and PSNR. In the simulations of the second test picture, we also found that the optimal choice of $p$ still equals 1. However, for PSNR, the optimal value of damping parameter is much greater that the optimal $\eta$ for the first test picture. In both cases, a large value of damping parameter will not decrease the quality of denoised images significantly. Therefore, we recommend to set $p=1$ and relative large value of damping parameter, e.g. $\eta=300$, as the initial guess of parameters for our method in practice.

\begin{table}[!htb]
\label{TabDamping}
\caption{Results with different damped parameters $\eta$  and $p$; the table shows the SSIM value and PSNR value.}
   \begin{center}
   \begin{tabular}{c|ccccccc}
   \hline\hline
    $p\setminus \eta$ & 0.001  & 1  & 100  & 300 & 600 &1500  & 3000 \\
    \hline
    Picture (a) & SSIM \\
    1 & 0.108 &  0.108 & 0.538 & 0.549 & 0.548 & 0.548 & 0.548  \\
    1.5 & 0.116 &  0.116 & 0.114 & 0.530 & 0.532 & 0.532 & 0.532  \\
    2 & 0.032 &  0.032 & 0.032 & 0.032 & 0.032 & 0.081 & 0.135  \\
    \hline
    Picture (a) & PSNR \\
    1 & 14.22 &  16.83 & 18.01 & 18.37 & 18.38 & 18.18 & 17.24  \\
    1.5 & 12.16 &  14.62 & 15.51 & 15.78 & 15.78 & 15.63 & 15.70  \\
    2 & 10.85 &  13.28 & 14.11 & 14.36 & 14.37 & 14.22 & 14.28  \\
    \hline\hline
    Picture (b) & SSIM\\
    1 & 0.366 &  0.374 & 0.482 & 0.697 & 0.752 & 0.754 & 0.753  \\
    1.5 & 0.403 & 0.487 & 0.541 & 0.683 & 0.748 & 0.748 & 0.747  \\
    2 & 0.107 &  0.108 & 0.291 & 0.359 & 0.383 & 0.403 & 0.462  \\
    \hline
    Picture (b) & PSNR\\
    1 & 17.26 &  24.79 & 26.86 & 26.84 & 21.81 & 19.29 & 17.92  \\
    1.5 & 14.95 & 20.59 & 22.14 & 22.13 & 18.36 & 16.46 & 15.44  \\
    2 & 13.58 &  18.85 & 20.30 & 20.29 & 16.76 & 15.00 & 14.74  \\
    \hline\hline
  \end{tabular}
  \end{center}
\end{table}

\subsection{Comparison with other state-of-the-art methods}

In order to show the advantages of our algorithm over existing approaches, we solve the same problem by the following methods: Total Variation (TV), Modified Telegraph (MTele, \cite{cao}), Telegraph (Tele, \cite{Ratner2013}), Total Generalized Variation of the second order (TGV, \cite{bredies2010total,Setzer2011}), and Median Filtering (MF, \cite{Lim1990}). In our MF method, each output pixel contains the median value in a 3-by-3 neighborhood around the corresponding pixel in the input image. In this group of simulations, we set $\eta= 300$ and $\eta= 1500$ for the first and second test pictures respectively. Moreover, $p=1$ for both test pictures.

Firstly, we compare the number of iterations required to reach a certain SSIM or PSNR by the five different iterative/dynamical denoising algorithms. The results are displayed in Table 2. Moreover, the evolutions of the SSIM value and PSNR value with respect to iterations for each method are shown in Figure \ref{evolutionMethods}, where one can see that unlike other four methods, whose SSIM and PSNR value are almost monotonic increasing, the SSIM/PSNR value of SV-DDF are oscillating during the evolutions. However, the trend of the SSIM/PSNR value for SV-DDF is to be a increasing function. By numerical simulations, which is omitted here, we found that the more oscillations of the SSIM/PSNR value of SV-DDF occur, the smaller the damping parameter $\eta$ in the model (\ref{SecondOrderFlow}) is. This is an expected result due to the behaviour of damped Hamiltonian systems.

\begin{table}[H]
\label{ComparisonIterations}
\caption{Comparison of methods for a set final value of SSIM and PSNR.}
   \begin{center}
   \begin{tabular}{c|ccccc}
   \hline
    & SV-DDF  & TV  & MTele  & Tele  &  TGV \\
   \hline
    Test picture (a) & Iterations \\
    \hline
    Initial SSIM (Reached SSIM) &  \\
    0.108 (0.400) & 253 &  484 & 447 & 421  & 576    \\
    Initial PSNR (Reached PSNR) &  \\
    12.07 (15.74) & 227 &  592 & 476 & 468  & 524    \\
    \hline
    Test picture (b) & Iterations \\
    \hline
    Initial SSIM (Reached SSIM) &  \\
    0.367 (0.435) & 64 &  326 & 167 & 154  & 397    \\
    Initial PSNR (Reached PSNR) &  \\
    18.81 (24.00) & 67 &  492 & 129 & 124  & 502    \\
    \hline
  \end{tabular}
  \end{center}
\end{table}

\begin{figure}[!htb]
\centering
\subfigure[]{
\includegraphics[width=0.48\textwidth]{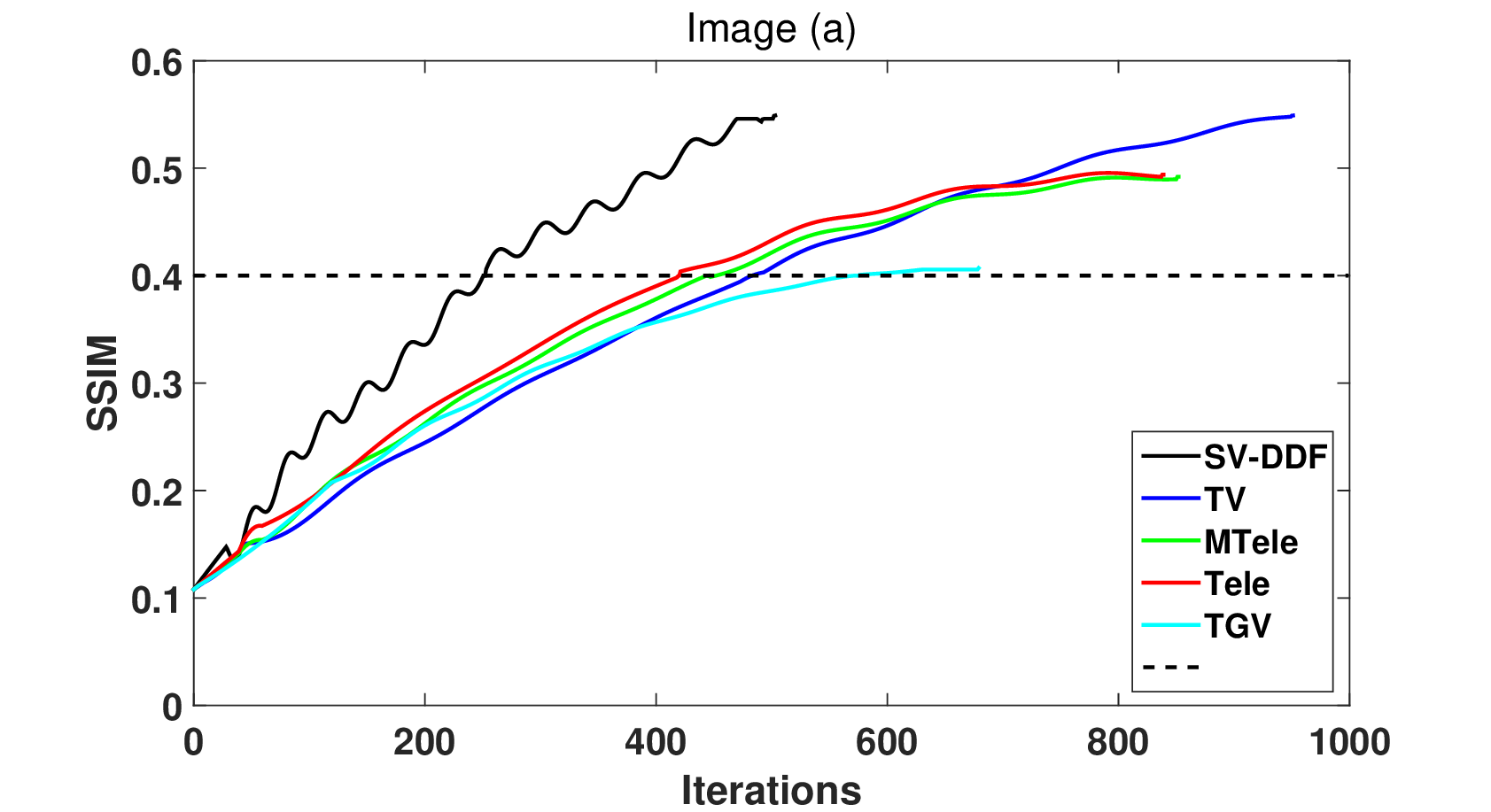}}
\subfigure[]{
\includegraphics[width=0.48\textwidth]{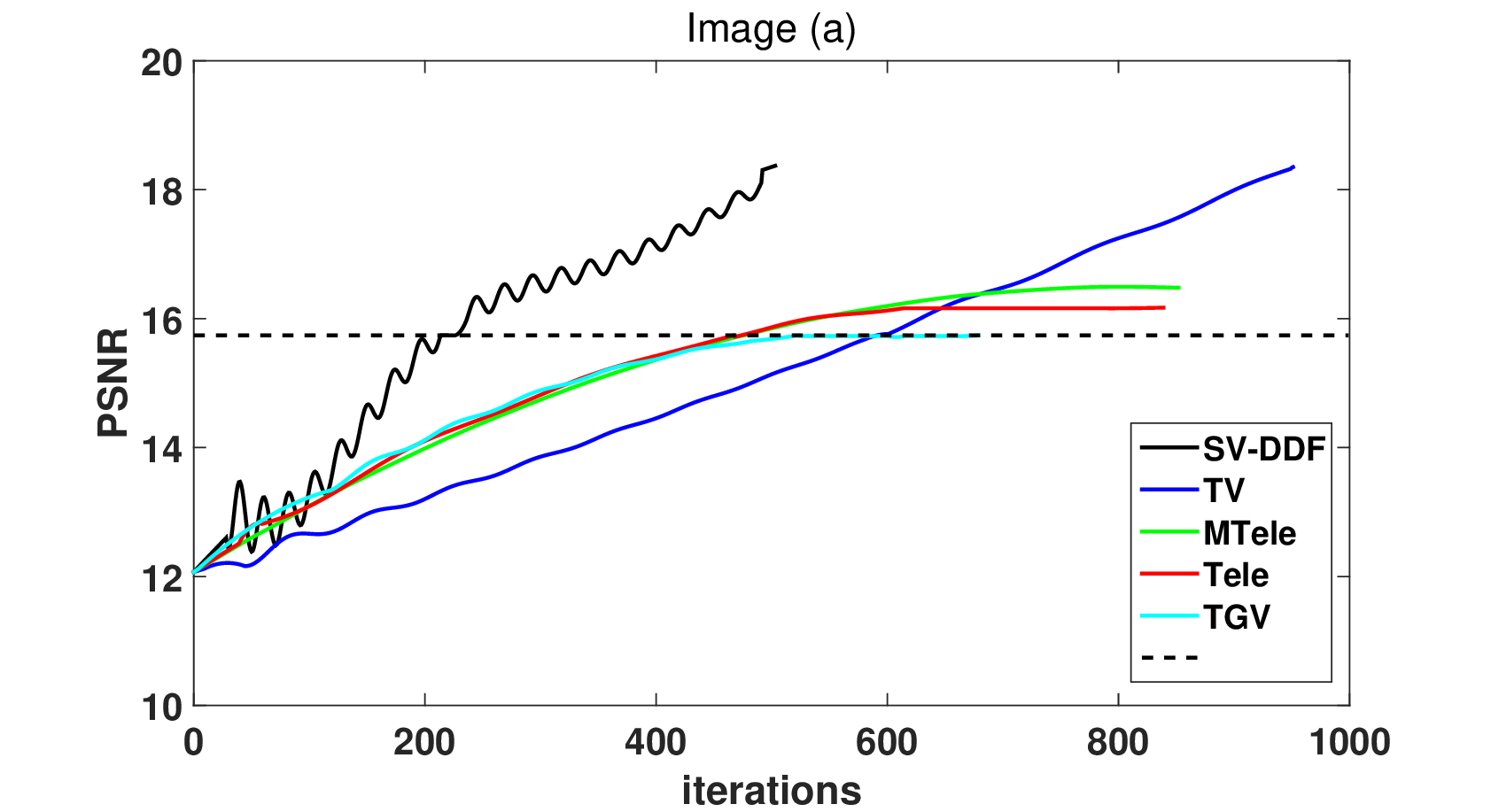}}
\subfigure[]{
\includegraphics[width=0.48\textwidth]{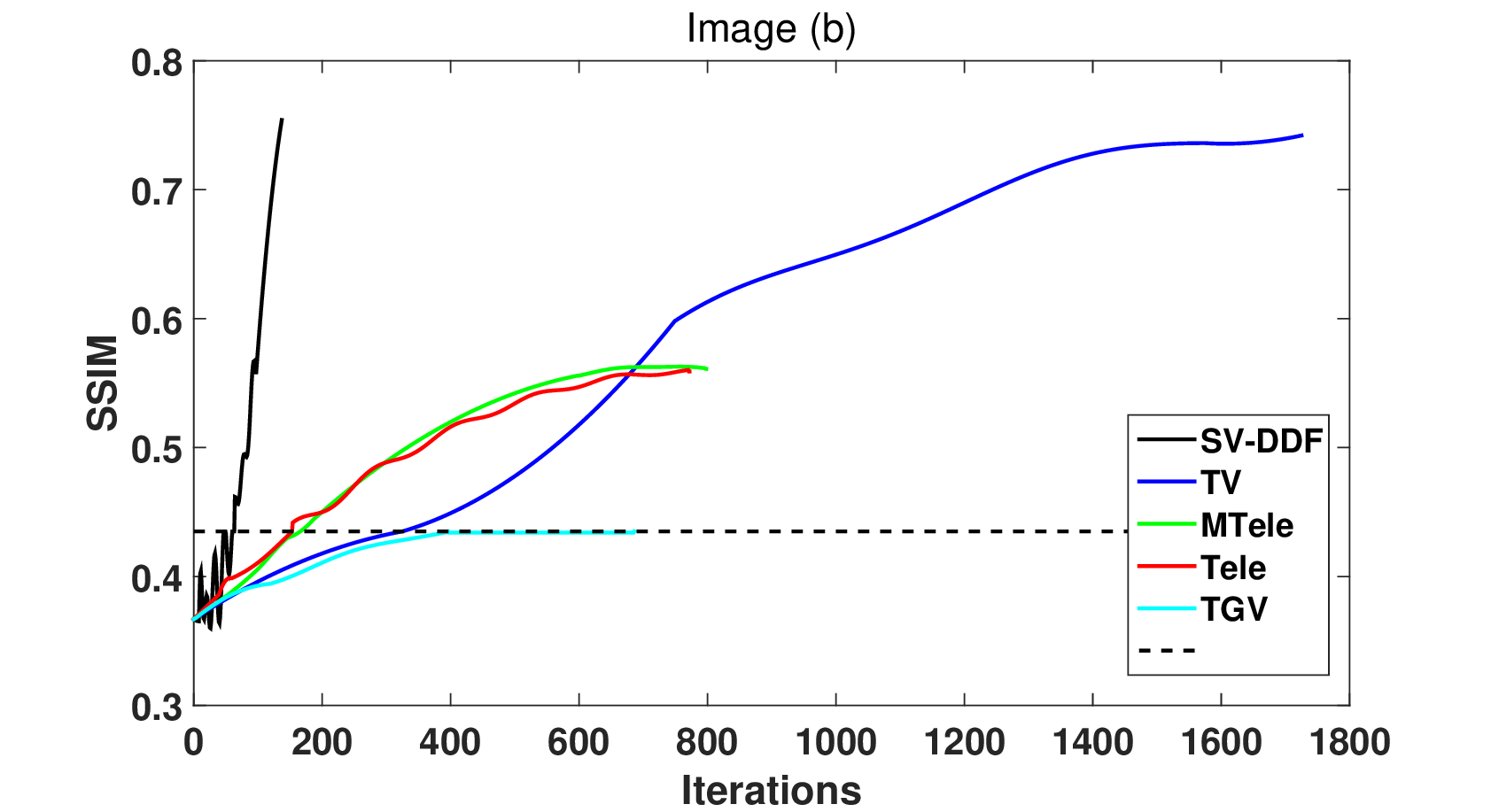}}
\subfigure[]{
\includegraphics[width=0.48\textwidth]{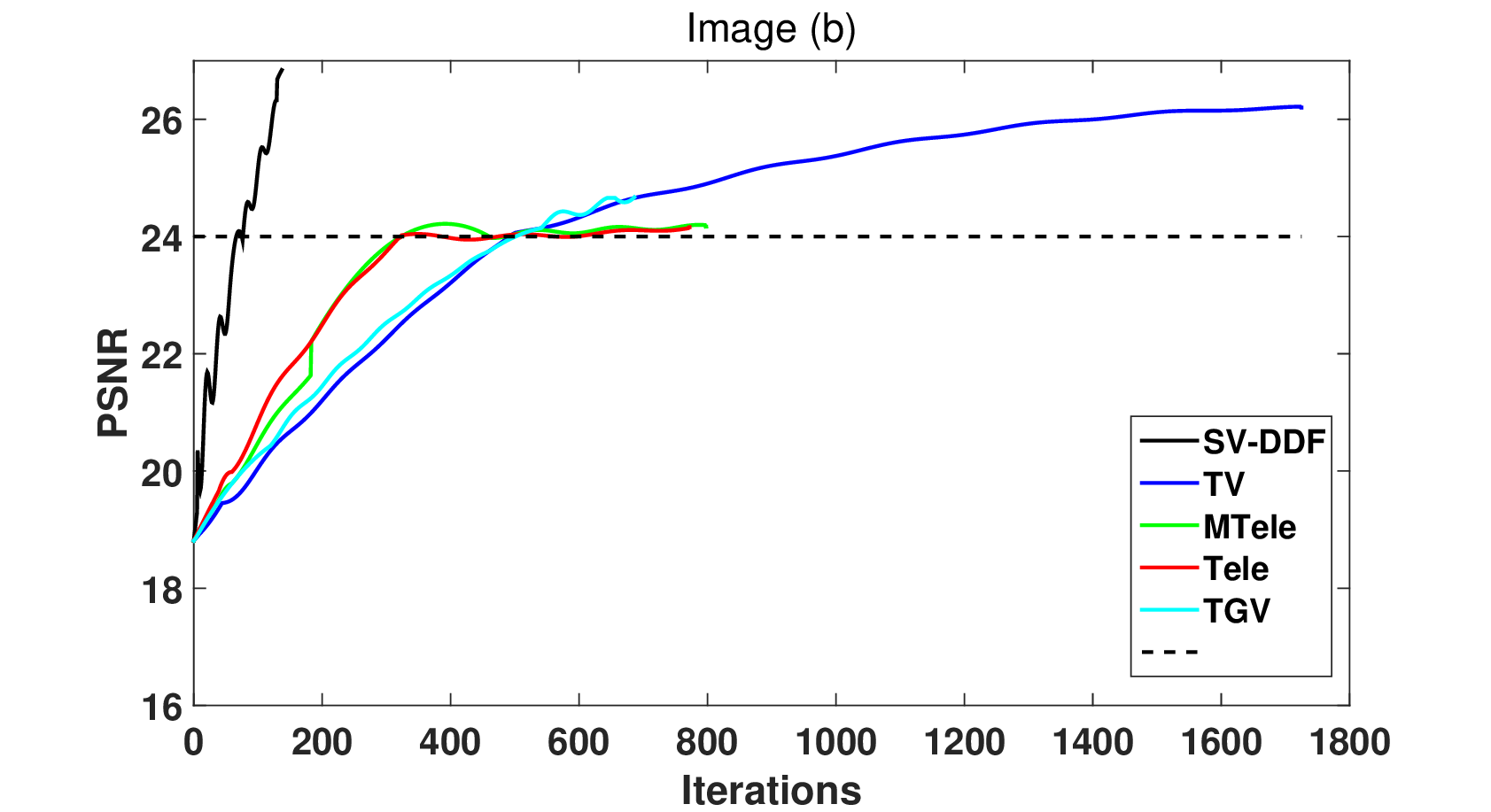}}
\caption{Evolutions of the SSIM value and PSNR value with respect to iterations for the five iterative denoising algorithms: SV-DDF, TV, MTele, Tele and TGV.}
\label{evolutionMethods}
\end{figure}

\renewcommand{\imsize}{0.25}
\begin{figure}[!t]
    \begin{tabularx}{\textwidth}{XX X}
      \toprule
    \end{tabularx}
    \includegraphics[clip, trim=0.8in 0.8in 0.8in 0.8in, width=\imsize\textwidth]{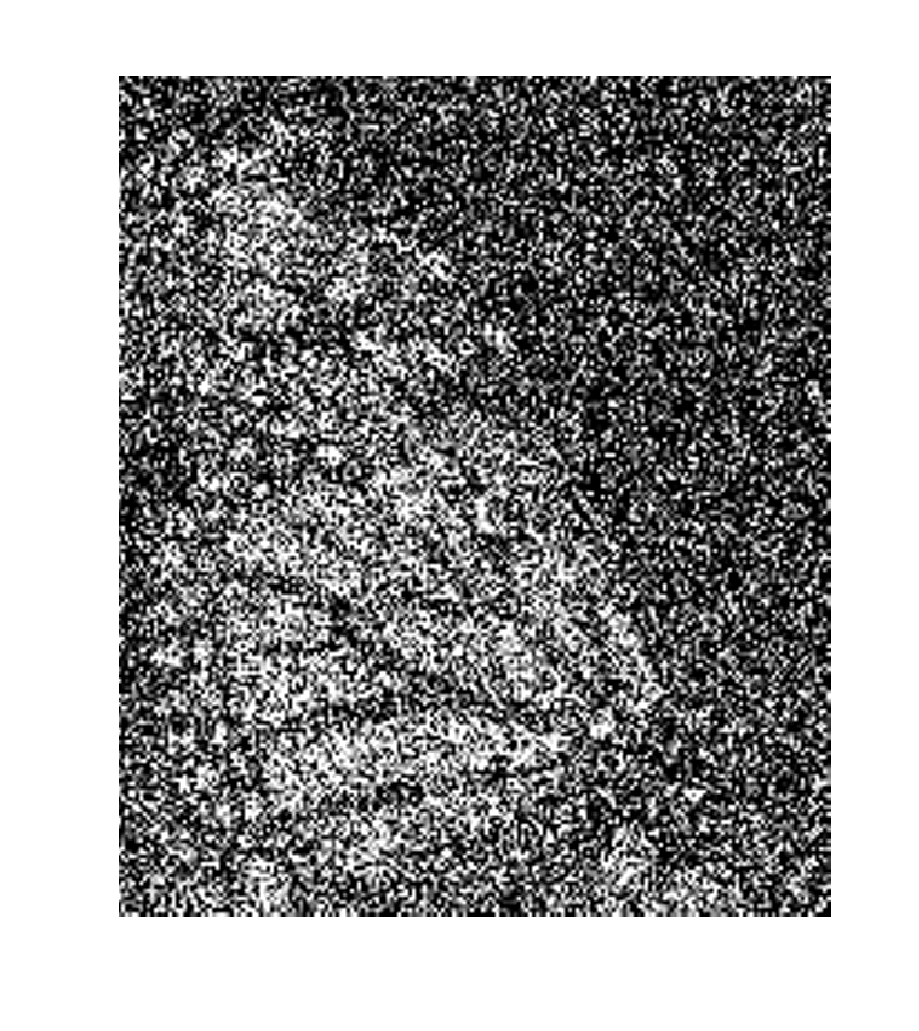}
    \includegraphics[clip, trim=0.8in 0.8in 0.8in 0.8in, width=\imsize\textwidth]{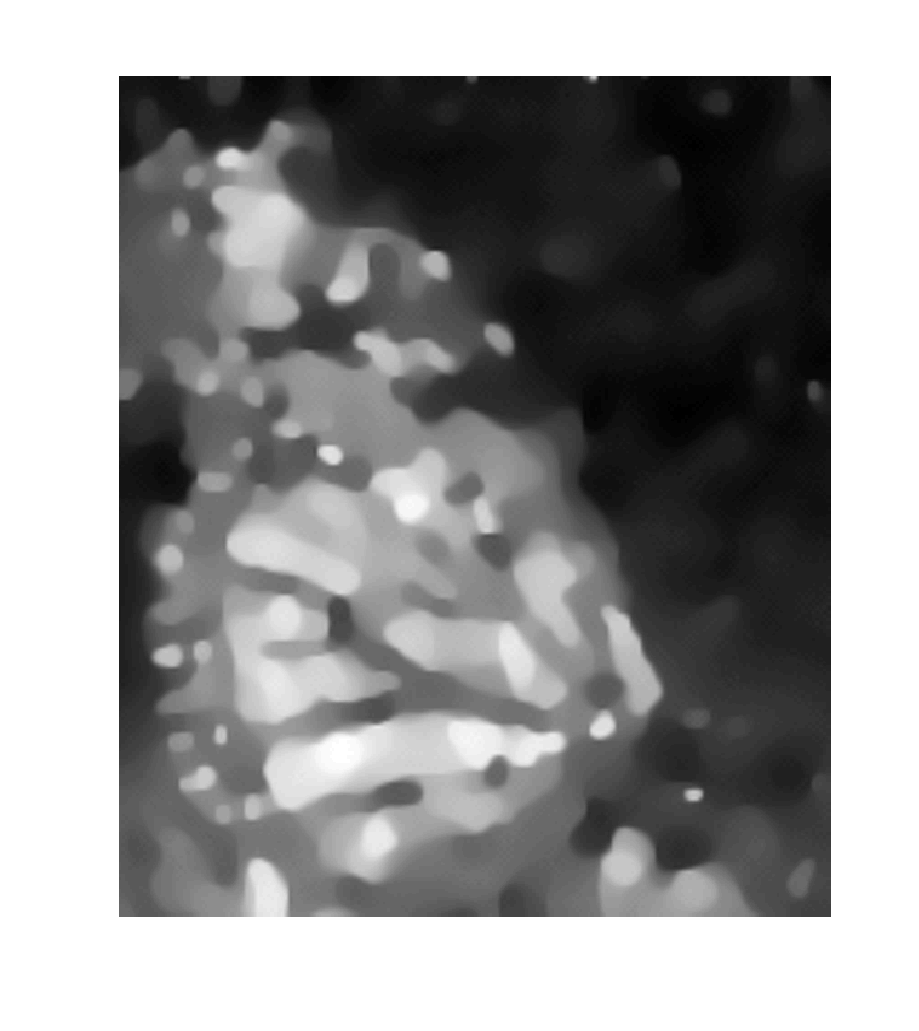}
    \includegraphics[clip, trim=0.8in 0.8in 0.8in 0.8in, width=\imsize\textwidth]{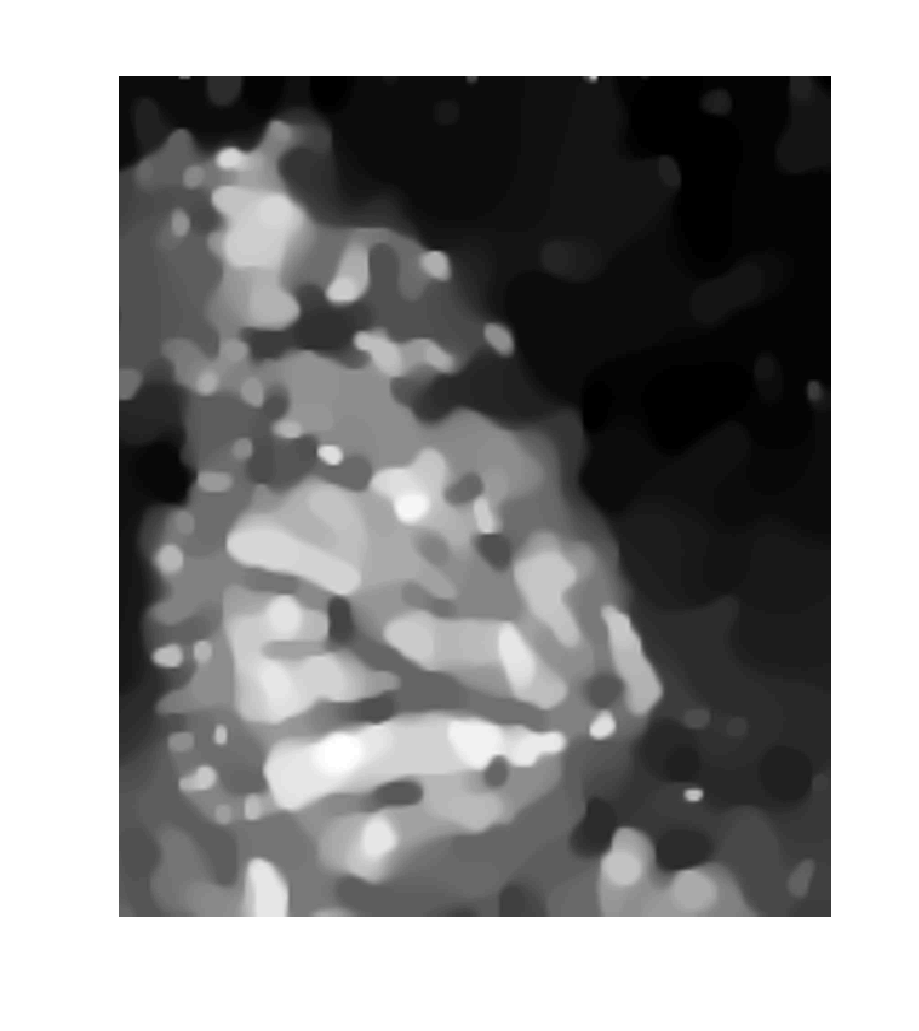}
    \begin{tabularx}{\textwidth}{XX XX}
      Noisy & SV-DDF (504 iterations) & TV (952 iterations) \\ 
      SSIM: 0.108 & SSIM: 0.549 & SSIM: 0.549 \\ 
      PSNR: 12.07 & PSNR: 18.37 & PSNR: 18.35 \\ 
      \midrule
    \end{tabularx}
    \includegraphics[clip, trim=0.8in 0.8in 0.8in 0.8in, width=\imsize\textwidth]{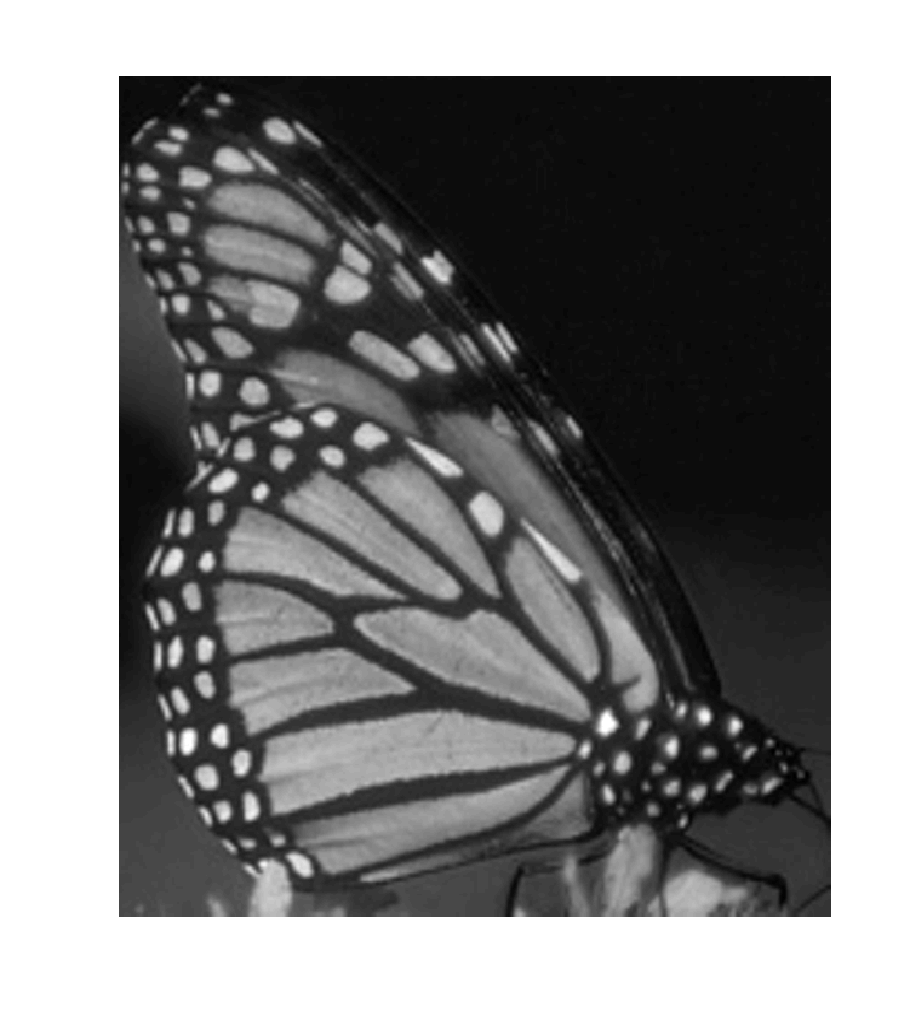}
    \includegraphics[clip, trim=0.8in 0.8in 0.8in 0.8in, width=\imsize\textwidth]{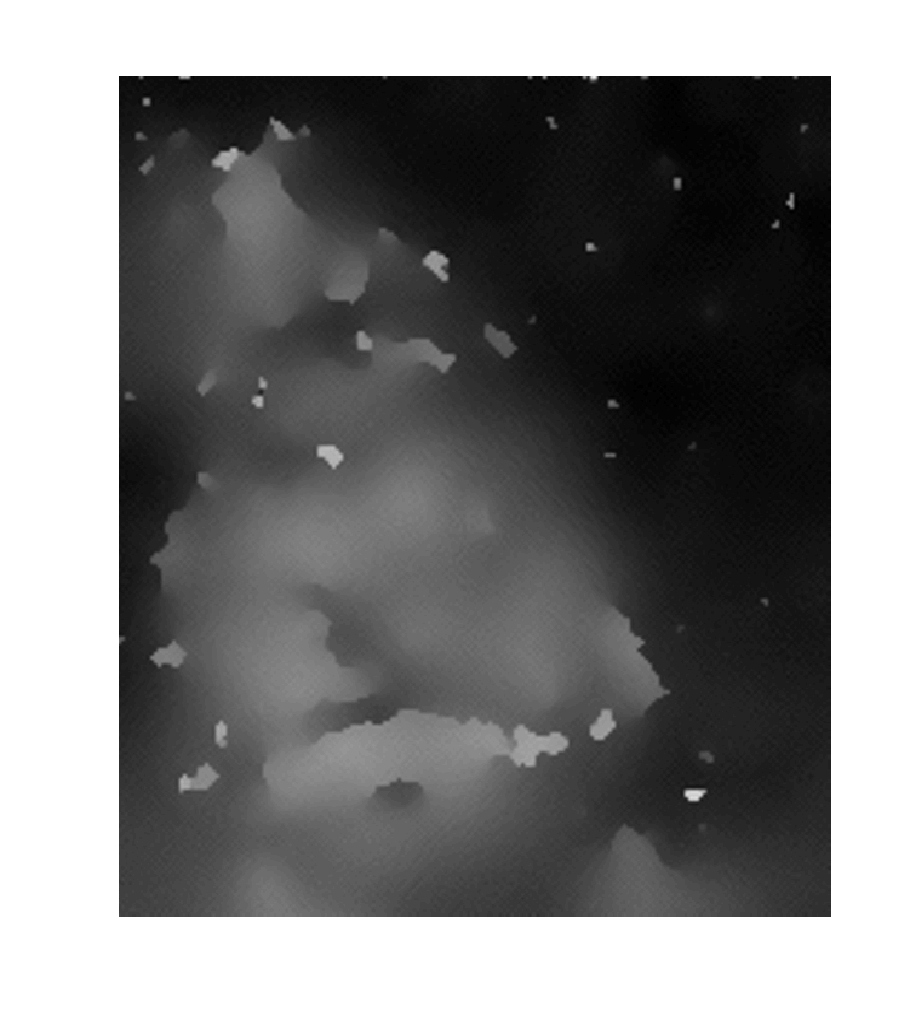}
    \includegraphics[clip, trim=0.8in 0.8in 0.8in 0.8in, width=\imsize\textwidth]{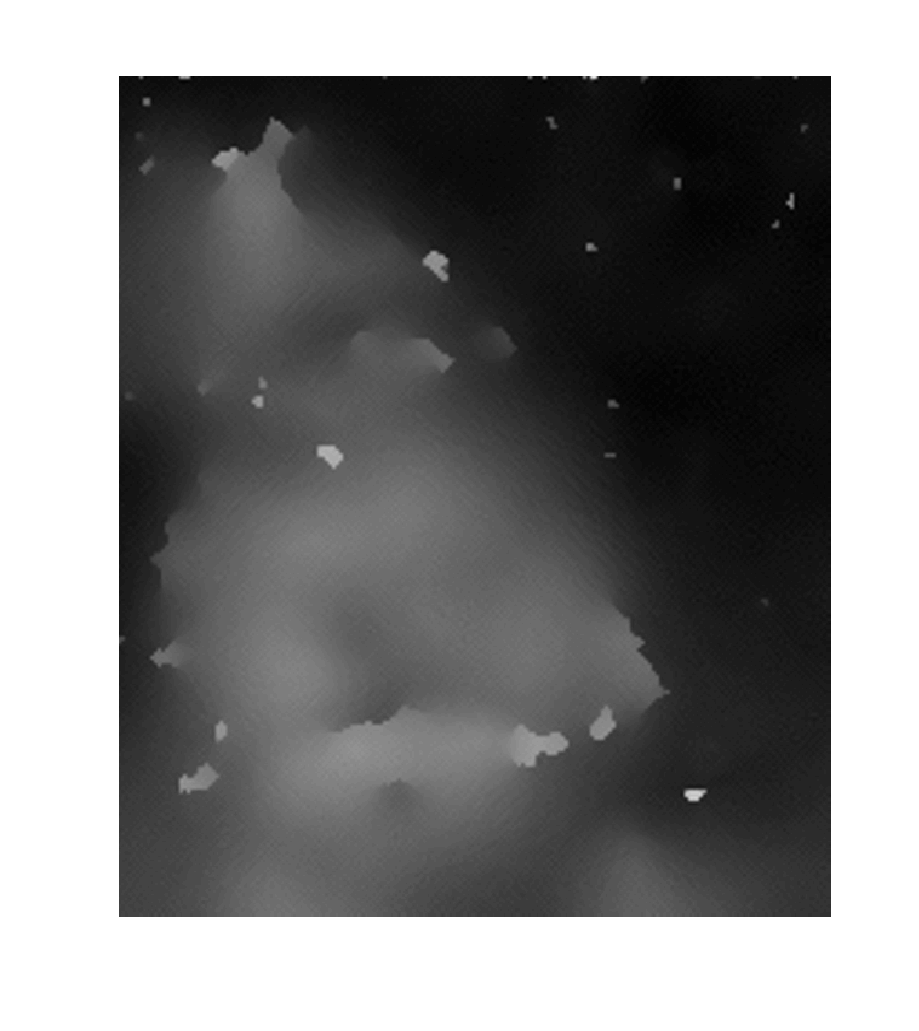}
    \begin{tabularx}{\textwidth}{XX XX}
      Original & MTele (853 iterations) & Tele (840 iterations) \\ 
      SSIM: & SSIM: 0.492 & SSIM: 0.494 \\ 
      PSNR: & PSNR: 16.48 & PSNR: 16.17 \\ 
      \bottomrule
    \end{tabularx}
    \includegraphics[clip, trim=0.3in 0.3in 0.2in 0.15in, width=\imsize\textwidth]{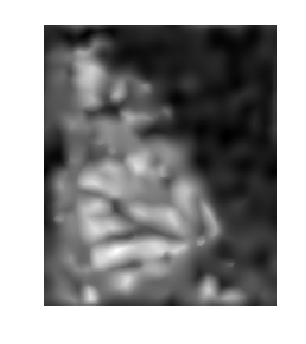}
    \qquad\qquad\quad
    \includegraphics[clip, trim=0.6in 0.65in 0.7in 0.4in, width=\imsize\textwidth]{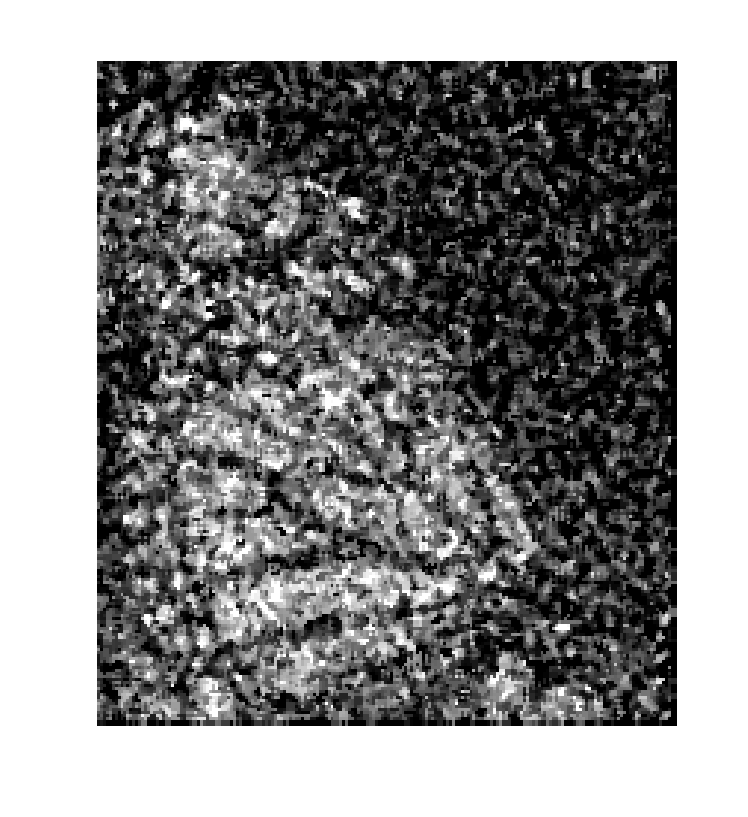} \\
    \begin{tabularx}{\textwidth}{XX XX}
      TGV (680 iterations) & MF &  \\ 
      SSIM: 0.407 & SSIM: 0.168 &  \\ 
      PSNR: 15.74 & PSNR: 14.02 &  \\ 
      \bottomrule
    \end{tabularx}
  \caption{Test picture (a): results for the methods SV-DDF, TV,  MTele, Tele, TGV and MF. }
  \label{fig:ssim}
  \vspace{-5mm}
\end{figure}

Next, we compare the qualities (the value of SSIM and PSNR) of the denoised images by the six mentioned algorithms. The results for two different types of noisy images are displayed in figures \ref{fig:ssim} and  \ref{fig:ssim2} respectively, where in each figure the test degraded image is given in the first picture, while the exact image and the denoised images are displayed in the last seven pictures. We see that for text image (a), the SV-DDF, as well as the TV method, gives the highest SSIM value 0.549, and also gives a smoother result than the other methods. However, the PSNR value of SV-DDF is slightly higher than the PSNR value of TV. Moreover, the TV method requires much more iterations to reach the result. For test image (b), the SV-DDF method presents the best results among all methods both regarding quality measures and iterations. However, as shown in figure \ref{fig:ssim2}, for image (b), like the TV method, our SV-DDF also exhibit the staircasing phenomenon. It is well known that the TGV methods do not lead to a staircasing effect, which motivates us to develop a new second order flow with TGV gradient, i.e.
\begin{equation}\label{GTVflow}
u_{tt} + \eta u_t  - \textmd{TGV}^2_\alpha (u) = 0, \quad u(x,0) = u_0(x), u_t(x,0) = 0  \textmd{~in~} \  \Omega,
\end{equation}
where the definition of $\textmd{TGV}^2_\alpha(\cdot)$, i.e. the total generalized bounded variation of order 2 with weight $\alpha$, can be found in (\cite{bredies2010total,Setzer2011}). Similar as SV-DDF in (\ref{symplectic}), one can propose a discretized version of second order flow (\ref{GTVflow}) (we denote it as SV-DDF-TGV). The result of SV-DDF-TGV for nosiy image (b) is displayed in the last picture of Figure \ref{fig:ssim2}, where we see that there is no staircase artifact for SV-DDF-TGV. However, the quality (both of value of SSIM and PSNR) of SV-DDF-TGV is much worse than the quality of the original SV-DDF method. Moreover, the well-posedness of the second order flow (\ref{GTVflow}), i.e. the existence, uniqueness, and regularization property of DF, are still open questions.

\renewcommand{\imsize}{0.25}
\begin{figure}[!tbh]
    \begin{tabularx}{\textwidth}{XX X}
      \toprule
    \end{tabularx}
    \includegraphics[clip, trim=0.8in 0.8in 0.8in 0.7in, width=\imsize\textwidth]{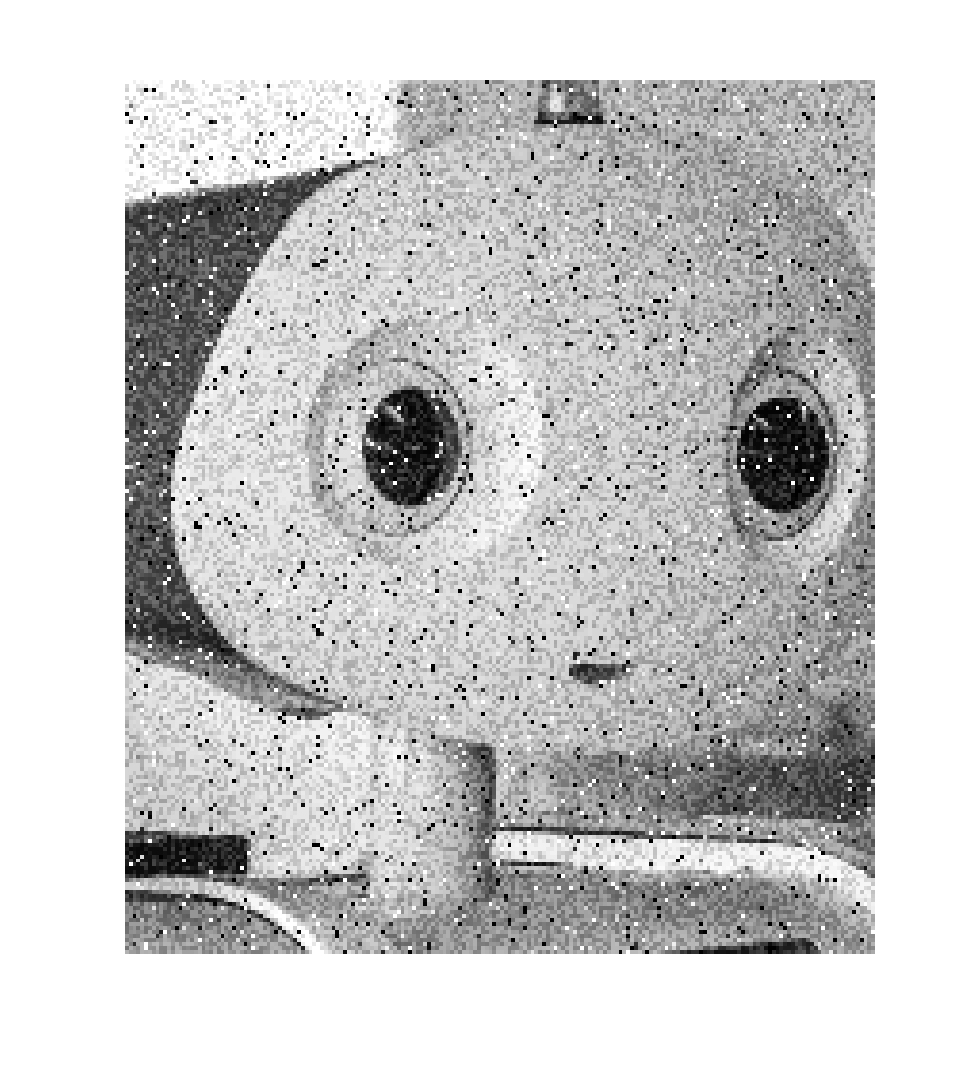}
    \includegraphics[clip, trim=0.8in 0.8in 0.8in 0.8in, width=\imsize\textwidth]{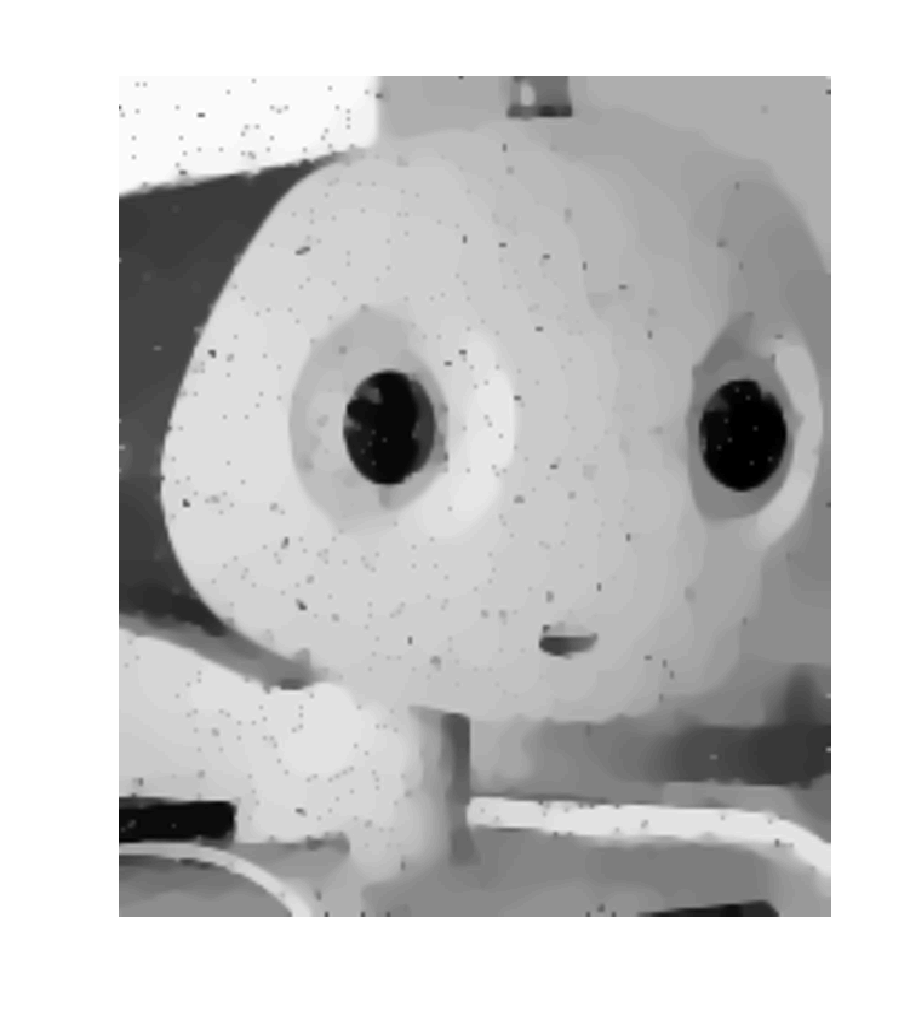}
    \includegraphics[clip, trim=0.8in 0.8in 0.8in 0.8in, width=\imsize\textwidth]{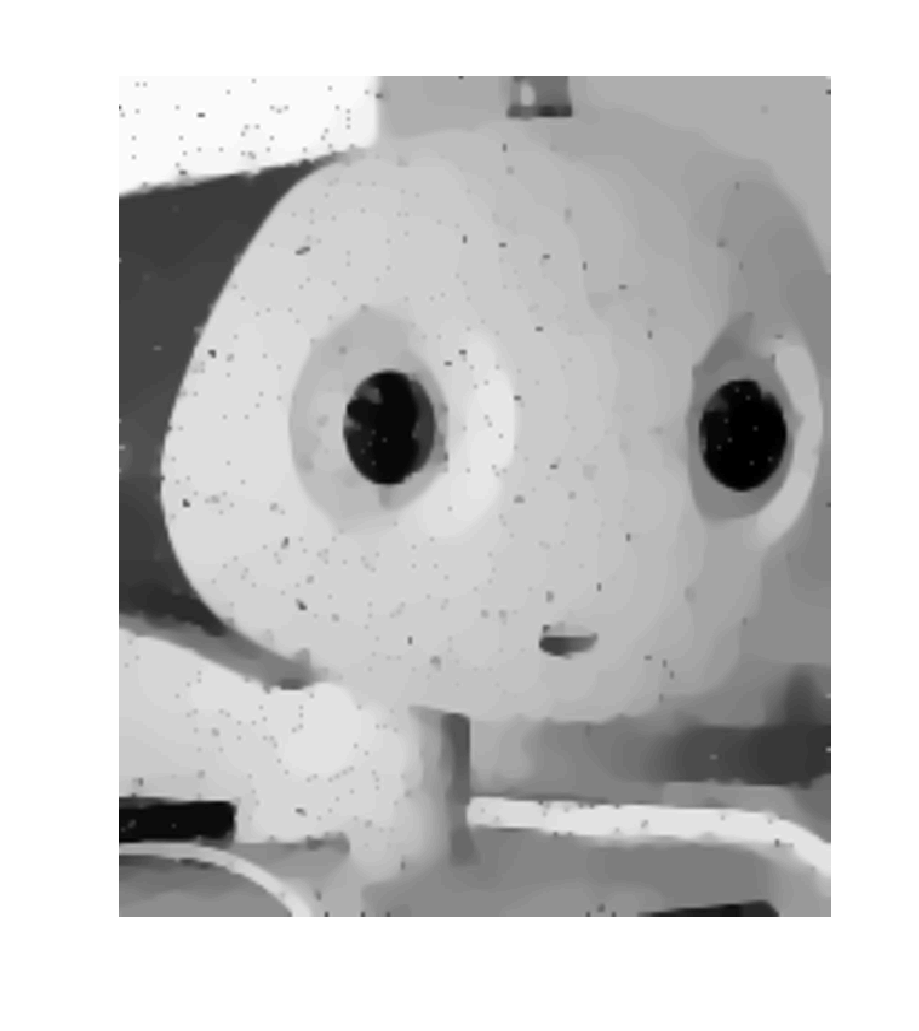}
    \begin{tabularx}{\textwidth}{XX XX}
      Noisy & SV-DDF (138 iterations) & TV (1726 iterations) \\ 
      SSIM: 0.367 & SSIM: 0.754 & SSIM: 0.742 \\ 
      PSNR: 18.81 & PSNR: 26.84 & PSNR: 26.20 \\ 
      \midrule
    \end{tabularx}
    \includegraphics[clip, trim=0.8in 0.8in 0.8in 0.7in, width=\imsize\textwidth]{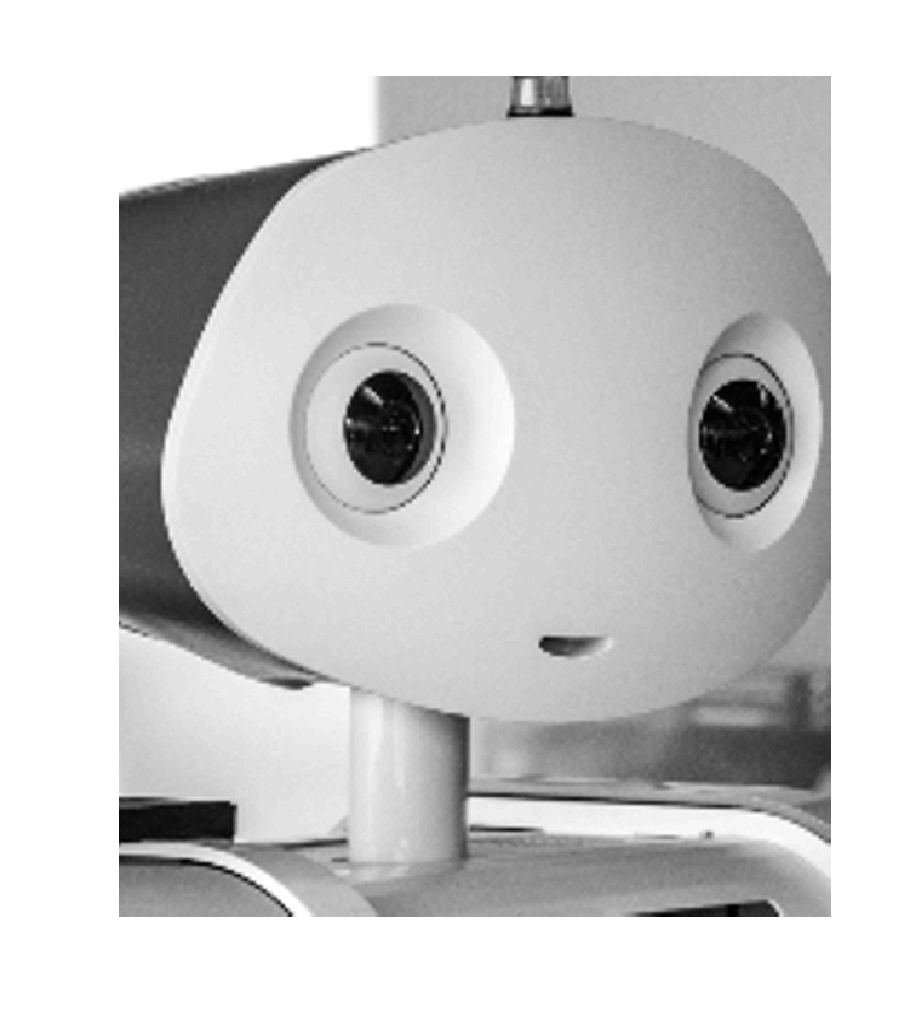}
    \includegraphics[clip, trim=0.8in 0.8in 0.8in 0.8in, width=\imsize\textwidth]{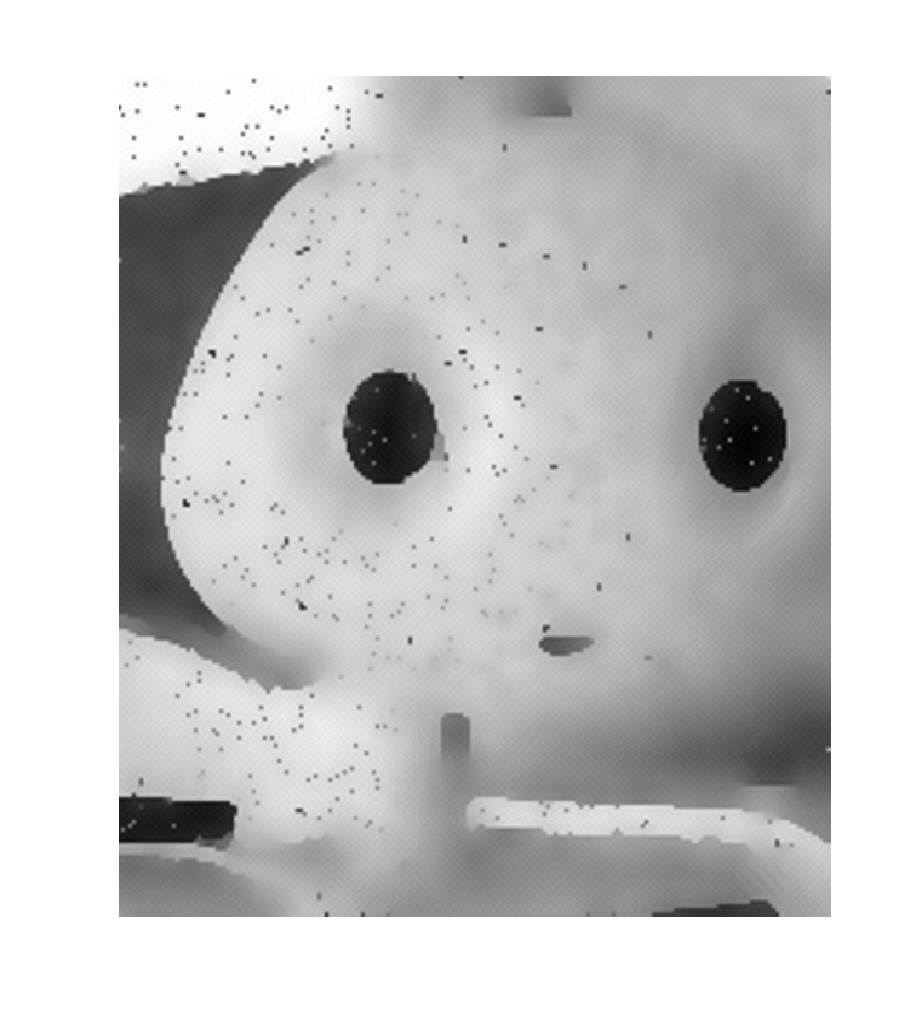}
    \includegraphics[clip, trim=0.8in 0.8in 0.8in 0.8in, width=\imsize\textwidth]{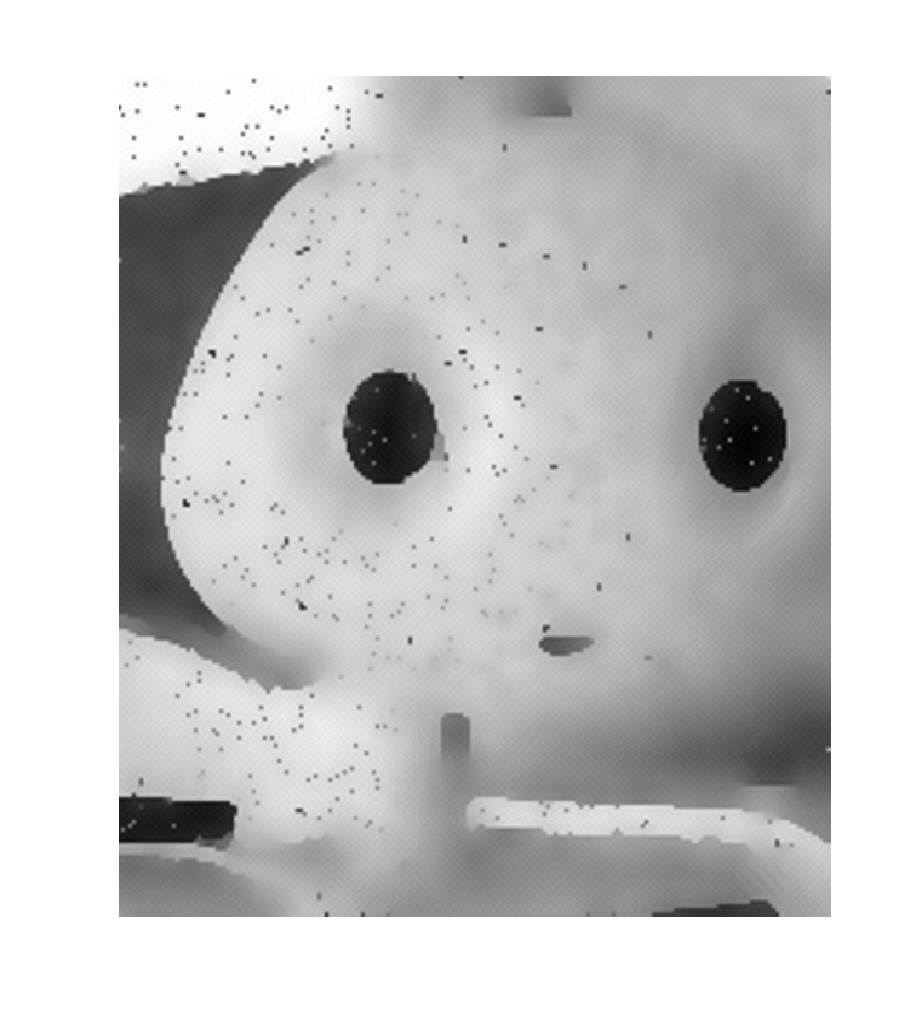}
    \begin{tabularx}{\textwidth}{XX XX}
      Original & MTele (799 iterations) & Tele (773 iterations) \\ 
      SSIM: & SSIM: 0.561 & SSIM: 0.559 \\ 
      PSNR: & PSNR: 24.17 & PSNR: 24.16 \\ 
      \bottomrule
    \end{tabularx}
    \includegraphics[clip, trim=0.6in 0.5in 0.4in 0.2in, width=\imsize\textwidth]{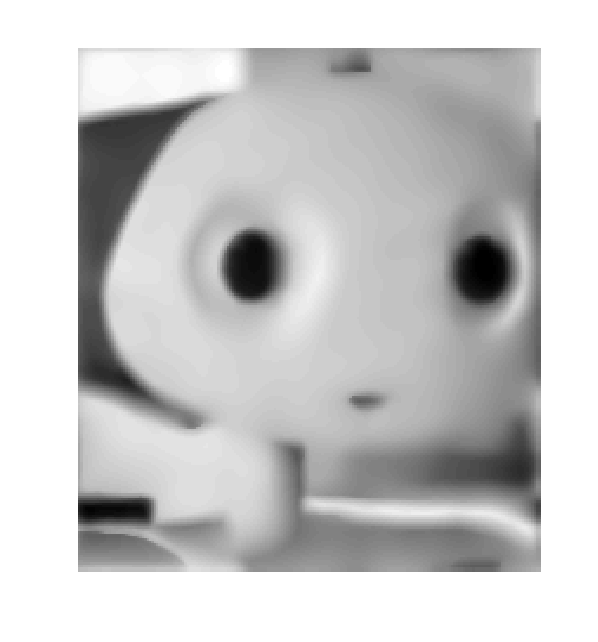}
    \includegraphics[clip, trim=0.6in 0.5in 0.4in 0.25in, width=\imsize\textwidth]{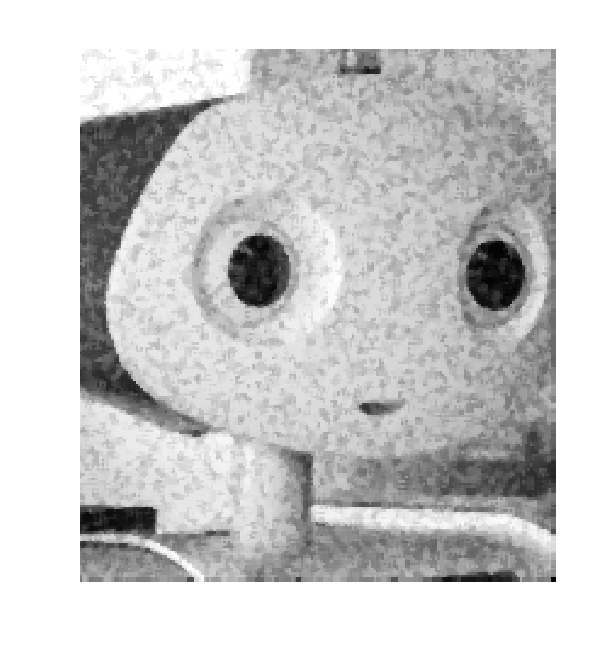}
    \includegraphics[clip, trim=0.3in 0.3in 0.2in 0.2in, width=\imsize\textwidth]{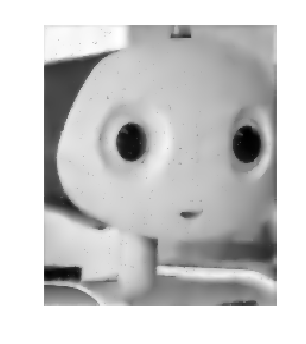}
    \begin{tabularx}{\textwidth}{XX XX}
      TGV (687 iterations) & MF & \setlength{\parindent}{-1em} SV-DDF-TGV (87 iterations) \\ 
      SSIM: 0.435 & SSIM: 0.282 & SSIM: 0.503 \\ 
      PSNR: 24.67 & PSNR: 23.82 & PSNR: 25.53 \\ 
      \bottomrule
    \end{tabularx}
  \caption{Test picture (b): results for the methods SV-DDF, TV,  MTele, Tele, TGV, MF and SV-DDF-TGV. }
  \label{fig:ssim2}
  \vspace{-5mm}
\end{figure}


\section{Conclusion.}
In this paper, we introduce a new image denoising model~-- the damped flow. The existence and uniqueness of the solution to the model are both proven under certain assumptions. For the numerical implementation, based on the St\"{o}rmer-Verlet method, a discrete damped flow, SV-DDF, is developed. The convergence of SV-DDF is discussed as well. A numerical algorithm with automatic stopping criterion is provided. A comparison with three existing methods show that the SV-DDF appears to be very competitive with respect to its image denoising capabilities and its acceleration affect. Obviously, beside in the $p$-Dirichlet energy denoising, the damped flow can also be used for solving a general non-linear optimization problem of the type (\ref{energyEu}), see e.g. the denoising model (\ref{GTVflow}). Finally, we emphasis that the aim of the paper is to introduce the SV-DDF method to image denoising. Since it is comparable to the conventional TV method, it is a promising approach which merits further theoretical and numerical development as well as more extensive comparison to state-of-the-art methods.

\section*{Acknowledgment}

We express our gratitude to the associate editor and two anonymous reviewers whose valuable comments and suggestions lead to an improvement of the manuscript.

The work of Y. Zhang is supported by the Alexander von Humboldt foundation through a postdoctoral researcher fellowship.


%


\end{document}